\documentclass[11pt,reqno]{amsart}
\usepackage{inputenc}
\usepackage{enumerate}
\usepackage{amssymb}
\usepackage{amsmath, amssymb}
\usepackage{mathrsfs}
\usepackage{amsthm}
\usepackage{tikz-cd}
\usepackage{mathpazo}
\usepackage{extarrows}
\usepackage{color}
\usepackage{setspace}
\usepackage[colorlinks,linkcolor = blue,anchorcolor = red,urlcolor = blue,citecolor = blue]{hyperref}
\usepackage{cleveref}
\usepackage{graphicx, subfig}
\usepackage[square, comma, sort&compress, numbers]{natbib}
\usepackage{amsaddr}

\theoremstyle{definition}
\newtheorem{Def}{Definition}[section]
\newtheorem{Exa}[Def]{Example}
\newtheorem{Rem}[Def]{Remark}

\theoremstyle{plain}
\newtheorem{Thm}[Def]{Theorem}
\newtheorem{Lem}[Def]{Lemma}
\newtheorem{Pro}[Def]{Proposition}
\newtheorem{Cor}[Def]{Corollary}

\newtheorem{Que}[Def]{Question}

\usepackage[left=2.6cm, right=2.6cm, top=3cm, bottom=3cm]{geometry}
\usepackage[all]{xy}
\usepackage{bm}

\def\IN{\mathbb N}\def\IR{\mathbb R}\def\IC{\mathbb C}\def\IA{\mathbb A}\def\IZ{\mathbb Z}\def\II{\mathbb I}

\def\B{\mathcal B}\def\E{\mathcal E}\def\G{\mathcal G}\def\K{\mathcal K}\def\P{\mathcal P}\def\L{\mathcal L}\def\H{\mathcal H}\def\I{\mathcal I}

\def\supp{\textup{supp}}
\def\ind{\textup{Ind}}
\def\prop{\textup{Prop}}
\def\diam{\textup{diam}}
\def\Cl{\textup{Cliff}_{\IC}}

\def\ox{\otimes}
\def\wh{\widehat}

\def\ox{\otimes}

\author[L.~Guo]{\vspace{-3ex}Liang Guo$^*$}
\address[L.~Guo]{\vspace{-2ex} Shanghai Institute for Mathematics and Interdisciplinary Sciences (SIMIS), Shanghai,
200433, China.\\
Research Institute of Intelligent Complex Systems, Fudan University, Shanghai, 200433, China\\ Email: liangguo@simis.cn}
\thanks{$^*$ Corresponding author: liangguo@simis.cn}

\author[K.~Li]{\vspace{-5ex}Kang Li}
\address[K.~Li]{\vspace{-2ex} Centre for Mathematical Sciences, Lund University, Box 118, SE-221 00 Lund, Sweden\\Email: kang.li@math.lth.se}

\author[Q.~Wang]{\vspace{-5ex}Qin Wang}
\address[Q.~Wang]{\vspace{-2ex} Research Center for Operator Algebras, School of Mathematical Sciences, East China Normal University, Shanghai, 200241, P.~R.~China.\\Email: qwang@math.ecnu.edu.cn}

\subjclass[2020]{19K56, 47L20}

\keywords{Coarse geometry, Roe algebra, Ghost ideal, $K$-theory, $\ell^p$-spaces}

\thanks{}

\title{$K$-theory of ghostly ideals for $\ell^p$-coarsely embeddable spaces}
\date{}

\begin{document}

\begin{abstract}
Ghostly ideals are among the most mysterious objects in coarse index theory. In this paper we show that if a metric space $X$ with bounded geometry admits a coarse embedding into an $\ell^p$-space ($1 \le p < \infty$), then the canonical inclusion from any geometric ideal to the corresponding ghostly ideal induces an isomorphism in $K$-theory. 
As consequences, we deduce that such spaces satisfy the relative coarse Baum-Connes conjectures introduced in \cite{GWZ2025}, as well as the operator norm localization property for finite rank projections ($ONL_{\mathcal P_{Fin}}$) as introduced in \cite{BFV2024}.
\end{abstract}

\maketitle

\tableofcontents

\section{Introduction}

The notion of ghosts was introduced by Guoliang Yu in his study of the coarse Baum-Connes conjecture and it plays a very central role in both coarse geometry and operator algebra. Ghosts are among the most mysterious operators in Roe algebras: a ghost is an infinite matrix such that its matrix entries vanish at infinity, yet it may fail to be a compact operator. More precisely, let $C^*(X)$ denote the Roe algebra of a metric space $X$ with \emph{bounded geometry}, i.e., there is a uniform bound on the cardinalities of all closed balls of a fixed radius (see \Cref{def: Roe algebra}). An operator $T \in C^*(X)$ can be written as an $X$-by-$X$ matrix such that each matrix entry is a compact operator, and $T$ is called a \emph{ghost} if $\|T_{xy}\| \to 0$ as $x, y \to \infty$. However, despite this local vanishing behavior of its matrix entries, the operator $T$ itself is not necessarily compact. In other words, the operator norm $\|\chi_{K_R}T\chi_{K_R}\|$ may not tend to zero as $R \to \infty$, where $K_R = B(x_0,R)$ for any fixed base point $x_0\in X$ and $R>0$. Consequently, a ghost is an object that is elusive locally at infinity but remains capturable globally. 

As mentioned previously, ghosts play a crucial role in the study of the coarse Baum-Connes conjecture by providing counterexamples. For a metric space $X$ with bounded geometry, there is a so-called coarse assembly map
$$\mu: \lim_{d\to\infty}K_*(P_d(X))\to K_*(C^*(X)).$$
The coarse Baum-Connes conjecture states that the map $\mu$ is an isomorphism of abelian groups. However, if $X$ is a sequence of expander graphs, there exists a \emph{non-compact} ghost projection in the Roe algebra $C^*(X)$. This projection defines an element in $K_0(C^*(X))$, which violates the conjecture because it lies outside the image of the assembly map $\mu$, as shown in \cite{HLS2002, WYI2012}. For generalized results extending from expander graphs to (measured) asymptotic expanders, see \cite{KLVZ2021, LSZ2023}. On the other hand, the most significant sufficient condition for the coarse Baum-Connes conjecture was established in the milestone work of G. Yu in 2000, who proved that the conjecture holds for any metric space of bounded geometry admitting a coarse embedding into a Hilbert space (see \cite{Yu2000}). This result established a crucial link between coarse geometry of Hilbert spaces and the analytic index theory.

In analogy with the closed ideal $\mathcal{K}$ of all compact operators, the collection of all ghosts in a Roe algebra forms a closed ideal $\mathcal{G}$ in the Roe algebra, referred to as the \emph{ghostly ideal}\footnote{It is also frequently referred to in the literature as the \emph{ghost ideal}.}. It turns out that the ghostly ideal $\mathcal{G}$ contains $\mathcal{K}$ and is in fact directly linked to the coarse Baum-Connes conjecture. Indeed, M.~Finn-Sell in \cite[Proposition~35]{FS2014} proved that
the inclusion map from the compact ideal $\mathcal{K}$ into the ghostly ideal $\mathcal{G}$ induces an isomorphism in $K$-theory if and only if the coarse assembly map $\mu$ is an isomorphism, provided the space $X$ admits a \emph{fibred} coarse embedding into a Hilbert space. The notion of fibred coarse embedding into Hilbert space was first introduced by Chen, Wang and Yu in \cite{CWY2013}, and it greatly generalizes Gromov’s notion of coarse embedding into Hilbert space. As consequences of \cite{Yu2000, FS2014}, if a space $X$ coarsely embeds into a Hilbert space, the inclusion map from the compact ideal $\mathcal{K}$ into the ghostly ideal $\mathcal{G}$ induces an isomorphism in $K$-theory (see \cite[Corollary~36]{FS2014}).

In recent years, there has been an uptick in interest in the study of general \emph{geometric ideal} $\mathcal{I}(X,U)$ and ghostly ideal $\mathcal{G}(X,U)$ associated with a specific invariant open subset $U$ in $\beta X$ (see \cite{WZ2023,WFZ2025}). Roughly speaking, for an invariant open set $U$ in $\beta X$, the corresponding ghostly ideal consists of operators whose matrix entries $T_{xy}$ vanish as $(x,y)$ tends to infinity along the ($\beta X\backslash U$)-direction, rather than requiring matrix entries to vanish in all directions at infinity.  We refer the reader to \Cref{def: geom vs ghost}. When $U=X$, the corresponding geometric ideal $\mathcal{I}(X,X)$ coincides exactly with the compact ideal $\mathcal{K}$, and the corresponding ghostly ideal $\mathcal{G}(X,X)$ coincides with the original ghostly ideal $\mathcal{G}$ of all ghost operators in the Roe algebra. On the other hand, we recover the Roe algebra when $U=\beta X$ in the sense that $\mathcal{I}(X,\beta X)=\mathcal{G}(X,\beta X)=C^*(X)$. Moreover, it is known that all ghostly ideals in the Roe algebra are geometric if and only if the underlying space $X$ has Property A (see e.g. \cite[Proposition~5.1]{WFZ2025}), and the canonical inclusion from any geometric ideal $\mathcal{I}(X,U)$ to the corresponding ghostly ideal $\mathcal{G}(X,U)$ induces an isomorphism in $K$-theory if the underlying space $X$ coarsely embeds into a Hilbert space (see \cite[Theorem~5.3]{WFZ2025}).

Progress on the coarse Baum-Connes conjecture remained limited until recently, when the connection between the coarse Baum-Connes conjecture and the coarse geometry of $\ell^p$-spaces ($1\le p<\infty$) became clearer. The principal achievements in this direction establish the injectivity of the coarse assembly map $\mu$ for spaces admitting a fibred coarse embedding into an $\ell^p$-space (see \cite{GLWZ2024}), and the bijectivity of $\mu$ for spaces admitting a coarse embedding into an $\ell^p$-space (see \cite{WXYZ2024}). A natural question arises in this perspective: does the $K$-theory isomorphism between the geometric ideal $\I(X,U)$ and the ghostly ideal $\G(X,U)$ persist for the space $X$ that coarsely embeds into an $\ell^p$-space? In this paper, we answer this question affirmatively (even for spaces coarsely embeddable into an $L^p$-space):

\begin{Thm}[\Cref{thm: main theorem} \& \Cref{thm: so do Lp}]\label{thm: 1.1}
Let $X$ be a metric space with bounded geometry. If $X$ admits a coarse embedding into some $\ell^p$-space (or $L^p$-space) with $p\in[1,\infty)$, then for any invariant open subset $U\subseteq\beta X$, the inclusion $i:\I(X,U)\to\G(X,U)$ induces an isomorphism in $K$-theory, i.e.,
$$i_*: K_*(\I(X,U))\to K_*(\G(X,U))$$
is an isomorphism of abelian groups, where $\I(X,U)$ and $\G(X,U)$ are the geometric and ghostly ideal associated with $U$ in the Roe algebra $C^*(X)$, respectively.
\end{Thm}

It should be emphasized that the proofs of the $K$-theory isomorphism between the geometric and ghostly ideals in \cite{WZ2023, WFZ2025} depend crucially on the fact that a-T-menable groupoids are $K$-amenable (see \cite{Tu1999}). Indeed, a metric space $X$ with bounded geometry is coarsely embeddable into a Hilbert space if and only if the associated coarse groupoid $G(X)$ is a-T-menable (see \cite[Theorem~5.4]{STY2002}). For a space $X$ coarsely embeddable into an $\ell^p$-space, whether its coarse groupoid $G(X)$ is $K$-amenable remains an open question. It is precisely this gap that necessitates a fundamentally different approach in this paper, departing from earlier methods.

Our method to prove \Cref{thm: 1.1} is to generalize the $\ell^p$-Dirac-dual-Dirac construction in \cite{WXYZ2024} from the extreme case where $U=\beta X$ to a general invariant open subset $U\subseteq \beta X$. The argument proceeds in three main stages:

\emph{1. Reduction to sparse subspaces}: Since the $\ell^p$-Dirac-dual-Dirac construction is primarily applicable to sparse spaces, we first reduce the proof of \Cref{thm: 1.1} to the sparse subspaces of $X$ (see \Cref{pro: reduction to sparse space}).

\emph{2. Restriction of the construction to ideals}: On a sparse subspace, the $\ell^p$-Dirac-dual-Dirac construction restricts to the ghostly and geometric ideals, respectively (see \Cref{pro: dirac dual dirac}). This allows us to further reduce the proof to a comparison between the geometric and ghostly ideals within a twisted Roe algebra.

\emph{3. Comparison within the twisted framework}: In the twisted Roe algebra, the distinction between elements in the geometric ideal and the corresponding ghostly ideal vanishes asymptotically (see \Cref{lem: twisted geo vs gho}), which completes the proof of our main theorem.

Several corollaries follow as direct applications of our main result. The first corollary concerns the so-called \emph{relative coarse Baum-Connes conjecture}, introduced in \cite{GWZ2025}. Its validity provides a local obstruction to the existence of Riemannian metrics with positive scalar curvature at infinity (see \cite{GWZ2025} for details). 

\begin{Cor}[\Cref{cor: relative CBC}]
Let $X$ be a metric space with bounded geometry that admits a coarse embedding into an $\ell^p$-space ($p\geq 1$). For any subspace $Y$ in $X$, the relative coarse Baum-Connes conjecture holds for $(X, Y)$. 
\end{Cor}

The second involves $\mathrm{ONL}_{\mathcal{P}_{\mathrm{Fin}}}$—the \emph{operator norm localization property for equi-approximable finite-rank projections} introduced in \cite{BFV2024}. This property, which plays a role in rigidity problems and the structure of Cartan masas in uniform Roe algebras, is a weaker version of the classical \emph{operator norm localization property} \cite{CTWY2008}. It is known that $\mathrm{ONL}_{\mathcal{P}_{\mathrm{Fin}}}$ is equivalent to the condition that all ghost projections on every sparse subspace are compact. Nevertheless, a purely geometric characterization of $\mathrm{ONL}_{\mathcal{P}_{\mathrm{Fin}}}$ remains an open question.  In this paper, we employ higher index theory (see e.g. Lemma~\ref{lem: when ghost notequal cpt}) to verify this property for $\ell^p$ coarsely embeddable metric spaces.

\begin{Cor}[\Cref{cor: CElp implies ONLFin}]
Let $X$ be a metric space with bounded geometry. If $X$ admits a coarse embedding into an $\ell^p$-space, then $X$ has operator norm localization for equi-approximable finite-rank projections.
\end{Cor}

Finally, we extend our results from the reduced setting to the maximal setting. By combining our main theorem with the techniques in \cite{HIT2020, WXYZ2024}, we deduce the following result regarding the maximal coarse Baum-Connes conjecture:

\begin{Cor}[\Cref{cor: maximal corollary}]\label{cor1.4}
Let $X$ be a metric space with bounded geometry that admits a coarse embedding into an $\ell^p$-space ($p\geq 1$). \begin{itemize}
\item[(1)] The canonical quotient map $\pi: C^*_{\max}(X)\to C^*(X)$ between the maximal and reduced Roe algebras induces an isomorphism in $K$-theory, i.e.,
$$\pi_*: K_*(C^*_{\max}(X))\xrightarrow{\cong} K_*(C^*(X)).$$
\item[(2)] For any subspace $Y$ in $X$, the canonical quotient map $\pi: C^*_{\max,Y,\infty}(X)\to C^*_{Y,\infty}(X)$ between the maximal and reduced quotient Roe algebras induces an isomorphism in the level of $K$-theory, i.e.,
$$\pi_*: K_*(C^*_{\max,Y,\infty}(X))\xrightarrow{\cong} K_*(C^*_{Y,\infty}(X)).$$
\item[(3)] For any subspace $Y$ in $X$, the maximal relative coarse Baum-Connes conjecture holds for $(X, Y)$.
\end{itemize}\end{Cor}

Our results motivate the following open questions:

\begin{Que}
Let $X$ be a metric space with bounded geometry.
\begin{itemize}
\item [1.] If $X$ admits a coarse embedding into an $\ell^p$-space, is the coarse groupoid $G(X)$ necessarily $K$-amenable? Strong evidence for this can be found in Corollary~\ref{cor1.4}.

\item[2.] In \cite{DG2024}, a twisted version of the coarse Baum-Connes conjecture with coefficients is introduced. It is natural to ask whether the twisted coarse Baum-Connes conjecture with coefficients holds for $X$ if $X$ admits a coarse embedding into an $\ell^p$-space?
\end{itemize}
\end{Que}

The paper is organized as follows. In \Cref{sec: preliminary}, we recall some preliminaries on coarse geometry and ideal structures of Roe algebras. In \Cref{sec: proof}, we shall prove our main theorem in four steps. Finally, we discuss applications to the relative and maximal coarse Baum-Connes conjectures as well as the operator norm localization property for finite-rank projections in \Cref{sec: applications}.

\section{Geometric and ghostly ideals in Roe algebras}\label{sec: preliminary}

In this section, we shall recall some basic facts on coarse geometry and the ideal structure of Roe algebras.

\subsection{Coarse groupoid}

Throughout this paper, we shall always assume $(X,d)$ to be a metric space with bounded geometry, i.e., for any $R>0$,
$$\sup_{x\in X}\# B(x,R)<\infty,$$
where $B(x,R)=\{x'\in X\mid d(x,x')\leq R\}$. Denote by $\beta X$ the Stone-\v{C}ech compactification of $X$, which is the set of all ultrafilters on $X$ endowed with a topology making it a compact Hausdorff space such that $C(\beta X)\cong\ell^\infty(X)$. The reader is referred to \cite{Goldbring2022} for details of ultrafilters and Stone-\v{C}ech compactification. For each $R>0$, we define the \emph{$R$-diagonal} of $X\times X$ to be
$$\Delta_R=\{(x,y)\in X\times X\mid d(x,y)\leq R\}\subseteq X\times X.$$
When $R=0$, the set $\Delta_0$ is the diagonal in $X\times X$. A subset of an $R$-diagonal is called an \emph{entourage} (or a \emph{controlled set}). The set of all entourages, denoted by $\E_d$, is called the \emph{coarse structure} of $X$ associated with the metric. The \emph{coarse groupoid} of $X$, denoted by $G(X)$, is defined to be
$$G(X)=\bigcup_{R\geq 0}\overline{\Delta_R}\subseteq\beta(X\times X).$$
The space $\beta X$ is the unit space of $G(X)$. The projection onto the first (second, resp.) coordinate $r: X\times X\to X$, ($s: X\times X\to X$, resp.) extends continuously to the range map $r: G(X)\to\beta X$ (the source map $s: G(X)\to\beta X$, resp.). The reader is referred to \cite{STY2002} for details of the coarse groupoid. A subset $U\subset \beta X$ is called \emph{invariant} if $r(\gamma)\in U$ is equivalent to $s(\gamma)\in U$ for any $\gamma\in G(X)$. For a subset $Y\subseteq X$, we define
$$U_Y=\bigcup_{R\geq 0}\overline{B(Y,R)}\subseteq \beta X,$$
where $B(Y,R)=\bigcup_{y\in Y}B(y,R)$. It is direct to see that $U_Y$ is an invariant open set. We shall denote $F_Y=U_Y^c$, which is an invariant closed set.

A map $t: D\to R$ is called a \emph{partial translation} in $X$ if\begin{itemize}
\item[(1)] $D,R$ are subsets of $X$ and $t$ is a bijection between $D$ and $R$;
\item[(2)] there exists $M>0$ such that $\sup_{x\in D}d(x,t(x))\leq M$.
\end{itemize}
It is proved in \cite[Appendix C]{SpaWil2017} that for any $\gamma\in G(X)$ with $s(\gamma)=\alpha\in\beta X$, there exists a partial translation $t:D\to R$ in $X$ such that $\gamma=\lim_{x\to\alpha}(t(x),x)\in G(X)$.

For any map $f: X\to Y$, we shall still denote $f$ the extended continuous map $\beta X\to \beta Y$ with a little abuse of notation. For any subset $A\subseteq\beta X$, we denote $G(X)|^A_A=r^{-1}(A)\cap s^{-1}(A)$,  which is a subgroupoid. Let $Y\subseteq X$ be a subset. Then the canonical inclusion map $i: Y\to X$ extends to an inclusion $i:\beta Y\to\beta X$. Thus, one can identify $\beta Y$ with the closure of $Y$ in $\beta X$. As a result, we have that $G(Y)=G(X)|^{\beta Y}_{\beta Y}$. For any invariant subset $ U \subseteq \beta X$, we shall denote $U^{\beta Y}=\beta Y\cap U$, and it is direct to see that $U^{\beta Y}$ is $G(Y)$-invariant. A \emph{coarse ideal pair} $(X, U)$ includes a metric space $X$ and an invariant subset $U$. Two pairs $(X, U)$ and $(Y, V)$ are \emph{coarsely equivalent} if there exists a coarse equivalence $f: X\to Y$ such that $V$ is the smallest invariant open set containing $f(U)$. In this case, it is equivalent to say that $V$ is the minimal invariant open subset such that $U=f^{-1}(V)$.

\begin{Exa}
Let $Y\subseteq X$ be a subset. We shall denote $Y_R=B(Y,R)$ the $R$-neighbourhood of $Y$. Then the invariant open subset generated by $Y$, denoted by $U_Y$, is defined to be $U_Y=\bigcup_{R\geq 0}\overline{Y_R}\subseteq\beta X$. It is direct to check that $U_Y$ is indeed an invariant open subset.

Let $f: X\to Z$ be a coarse equivalence. Set $W$ to be the $G(Z)$-invariant open subset of $\beta Z$ generated by $f(Y)$. Then it is direct to check that $(X, U_Y)$ is coarsely equivalent to $(Z, W)$.
\end{Exa}

\subsection{Roe algebras and their ideals}

Fix $\H$ to be an infinite-dimensional, separable Hilbert space. Denote $\H_X=\ell^2(X)\ox\H$. Then there is a canonical representation of $C_0(X)$ on $\H_X$ by
$$f\cdot\delta_x\ox v=f(x)\delta_x\ox v$$
for $\delta_x\in\ell^2(X)$ and $v\in\H$. An operator $T\in\B(\H_X)$ can be viewed as an $X$-by-$X$ matrix with value in $\B(\H)$, where its $(x,y)$-th entry is given by:
$$\B(\H)\ni T_{xy}=\delta_xT\delta_y: \IC\delta_y\ox\H\to\IC\delta_x\ox\H,$$
where $\delta_x, \delta_y\in C_0(X)$ are Dirac functions on $x$ and $y$, respectively.

\begin{Def}\label{def: Roe algebra}
The \emph{algebraic Roe algebra} of $X$, denoted by $\IC[X,\H_X]$ (or $\IC[X]$ for simplicity), is a subalgebra of $\B(\H_X)$ consisting of all operators $T$ satisfying:
\begin{itemize}
\item $T$ has \emph{finite propagation} (or \emph{controlled propagation}), i.e., the \emph{support} of $T$
$$\supp(T)=\{(x,y)\in X\times X\mid T_{xy}\ne 0\},$$
is an entourage.
\item $T_{xy}\in\K(\H)$ for any $(x,y)\in X\times X$.
\end{itemize}
The completion of $\IC[X]$ in $\B(\H_X)$ is called the \emph{Roe algebra} of $X$, denoted by $C^*(X,\H_X)$ (or $C^*(X)$ for simplicity).
\end{Def}

We should remark that the definition of the Roe algebra does not depend on the choice of $\H$ up to a non-canonical isomorphism.
The following lemma is proved in \cite[Lemma 4.4]{STY2002}, which reveals that Roe algebras can be fundamentally understood as groupoid $C^*$-algebras. If we replace $\H$ with $\IC$, the resulting Roe algebra is called the \emph{uniform Roe algebra}, denoted by $C^*_u(X)$.

\begin{Lem}
Let X be a metric space with bounded geometry. There is a natural isomorphism
$$C^*(X)\cong \ell^{\infty}(X,\K(\H))\rtimes_r G(X).\qed$$
\end{Lem}

Let $\I$ be an ideal of $C^*(X)$. For any $T\in\I$, we define the $\varepsilon$-support of $T$ to be
$$\supp_\varepsilon(T)=\{(x,y)\in X\times X\mid \|T_{xy}\|\geq \varepsilon\}.$$
Since $T$ is in the Roe algebra, it is direct to see that $\supp_\varepsilon(T)$ is an entourage. We define
$$U(\I)=\bigcup_{T\in\I,\varepsilon>0}\overline{r(\supp_\varepsilon(T))}\subseteq \beta X.$$
Since $\I$ is an ideal, one can prove that $U(\I)$ is an invariant open subset, see \cite{CW2004} for a detailed discussion.

\begin{Rem}
For ideals in Roe algebras, we can analyze them through two components:
\begin{itemize}
\item Each ideal corresponds to a $G(X)$-invariant open set (determined by calculating operator supports as above);
\item After fixing the support, the coefficient part of each operator (i.e., the $\ell^\infty(X,\mathcal{K})$ component) further determines the rank distribution of operator entries within the ideal, as discussed in \cite{CW2006, WFZ2025}.
\end{itemize}
In this paper, we will not consider the $\ell^\infty(X,\mathcal{K})$ aspect. That is, we assume the ideal $\mathcal{I}$ has maximal rank distribution over its corresponding invariant open set, i.e., $\mathcal{I}$ contains all rank distributions supported on $U(\mathcal{I})$. We will not review the precise definition of rank distribution here; interested readers may refer to \cite{CW2006, WFZ2025} for details.
\end{Rem}

\begin{Def}\label{def: geom vs ghost}
Let $U\subseteq \beta X$ be an invariant open set. The \emph{geometric ideal} of $C^*(X)$ with respect to $U$, denoted by $\I(X, U)$, is defined to be the closed ideal generated by
$$\II[X,U]=\left\{T\in\IC[X]\ \big|\ \overline{r(\supp_{\varepsilon}(T))}\subseteq U\right\}.$$
The \emph{ghostly ideal} of $C^*(X)$ with respect to $U$, denoted by $\G(X,U)$, is defined to be
$$\G(X,U)=\left\{T\in C^*(X)\ \big|\ \overline{r(\supp_{\varepsilon}(T))}\subseteq U\right\}.$$
\end{Def}

One should notice that $\II[X, U]$ is dense in $\I(X, U)$ and is also an ideal of $\IC[X]$.
Note that the definitions of geometric ideals and ghostly ideals are very similar, but they have essential differences. The geometric ideal is constructed in the \emph{algebraic Roe algebra} by taking the part where the $\varepsilon$-support set lies in $U$ and then generating an ideal, while the ghostly ideal is obtained in the \emph{completed} Roe algebra by taking the part where the $\varepsilon$-support set lies in $U$. The essential difference lies in the order of taking the completion and imposing the support condition. Therefore, it is straightforward to see that the geometric ideal is a subalgebra of the ghostly ideal. Moreover, the ghostly ideal is larger than the corresponding geometric ideal, and the two ideals coincide only when the space has Yu's Property A (see \cite{WZ2023}).

Our main result of this paper is as follows.

\begin{Thm}\label{thm: main theorem}
Let $X$ be a metric space with bounded geometry. 
If $X$ admits a coarse embedding into some $\ell^p$-space with $p\in[1,\infty)$, then for any invariant open subset $U\subseteq \beta X$, the inclusion $i:\I(X,U)\to\G(X,U)$ induces an isomorphism on $K$-theory, i.e.,
\begin{equation}\label{eq: inclusion2}i_*: K_*(\I(X,U))\to K_*(\G(X,U))\end{equation}
is an isomorphism.
\end{Thm}

\begin{Rem}
By the following six-term exact sequence 
$$\begin{tikzcd}
K_0(\I(X,U)) \arrow[r]   & K_0(\G(X,U)) \arrow[r] & K_0\left(\frac{\G(X,U)}{\I(X,U)}\right) \arrow[d] \\
K_1\left(\frac{\G(X,U)}{\I(X,U)}\right) \arrow[u] & K_1(\G(X,U)) \arrow[l] & K_1(\I(X,U)), \arrow[l]  
\end{tikzcd}$$
we conclude that the map $i_*$ in \eqref{eq: inclusion2} is an isomorphism if and only if $K_*(\frac{\G(X,U)}{\I(X,U)})=0$. As a result, for an $\ell^p$-coarsely embeddable space $X$, one has that  $K_*(\frac{\G(X,U)}{\I(X,U)})=0$ for any invariant open subset $U\subseteq \beta X$.
\end{Rem}

As a result of P.~Nowak \cite{Nowak2006}, for $p\in[1,2]$, coarse embeddability into $\ell^p$ is equivalent to coarse embeddability into Hilbert space. In this case, the theorem is proved in \cite{WFZ2025}. 
It is worth mentioning that for $p, q\in[2,\infty)$, $\ell^p$ admits a coarse embedding into $\ell^q$ if and only if $p\leq q$, see \cite[Theorem 1.10]{MN2008}. Thus, coarse embedding into $\ell^p$-space is weaker than coarse embedding into Hilbert space when $p>2$. 
However, $\ell^p$ is not of bounded geometry and there is still no known example of a metric space with \emph{bounded geometry} that coarsely embeds into $\ell^p$ for $p> 2$ but can not coarsely embed into Hilbert space.

\subsection{Coarse invariance of ghostly/geometric ideals}

In this subsection, we shall show that the ghostly ideal is coarsely invariant. The reader is referred to \cite[Section 6]{JZ2025} for some relevant discussion.

\begin{Pro}
If $f: X\to Y$ is a coarse equivalence between the pair $(X,U)$ and $(Y,V)$, then $f$ induces an isometry $W_f:\H_X\to H_Y$ such that
$$Ad_{W_f}: \I(X, U)\to \I(Y,V)\quad\text{and}\quad Ad_{W_f}: \G(X, U)\to \G(Y,V).$$
In particular, $W_f$ can be chosen to be a unitary.
\end{Pro}

\begin{proof}
By \cite[Proposition 4.3.4 \& Proposition 4.3.5]{HIT2020}, one can find a covering isometry $W_f: \H_X\to\H_Y$ for $f$. For discrete metric spaces with bounded geometry, the construction of $W_f$ will be easier. We shall briefly recall the process here. 

Write $Z=f(X)$. We can take $A_z=f^{-1}(z)$. Choose a unitary $W_z: \ell^2(A_z)\ox\H \to\delta_z\ox\H$. Then
$$W_f=\oplus_{z\in Z}W_z:\bigoplus_{z\in Z}\ell^2(A_z)\ox\H=\H_X\to \ell^2(Z)\ox\H\subseteq\H_Y$$
is a covering isometry. To find a unitary, we can take a uniformly bounded Borel cover $\{B_z\}_{z\in Z}$ of $Y$ such that $z\in B_z$ for all $z\in Z$. Then $\{f^{-1}(B_z)=f^{-1}(z)\}_{z\in Z}$ is also a uniformly bounded Borel cover of $X$. Write $A_z=f^{-1}(z)$ for simplicity. For any $z\in Z$, take an arbitrary unitary $W_z: \ell^2(A_z)\ox\H \to\ell^2(B_z)\ox\H$. Then
$$W_f=\oplus_{z\in Z}W_z:\bigoplus_{z\in Z}\ell^2(A_z)\ox\H=\H_X\to\bigoplus_{z\in Z}\ell^2(B_z)\ox\H=\H_Y$$
is exactly the covering unitary we need.

For the geometric ideal, it suffices to prove $Ad_{W_f}(T)\in\II[Y, V]$ for any $T\in \II[X,U]$. By the bounded geometry condition of $X$, there exists $N>0$ such that $\#(f^{-1}(y))\leq N$. Fix $\varepsilon>0$. Let $A\in M_N(\K)$ be an $N$-by-$N$ matrix all of whose entries have norm smaller than $\varepsilon$. Then the norm of $A$ is smaller than $\varepsilon \cdot N$. For any $y\in Y$, we shall write $z_y\in Z$ such that $y\in B_{z_y}$. From the construction of $W_f$, we have that
$$\|Ad_{W_f}(T)_{yy'}\|\leq N\cdot \sup\{\|T_{xx'}\|\mid x\in A_{z_y},x'\in A_{z_{y'}}\}.$$
Thus, if $\|Ad_{W_f}(T)_{yy'}\|\geq\varepsilon$, there exists at least one pair $(x,x')\in A_{z_y}\times A_{z_{y'}}$ such that $\|T_{xx'}\|\geq \varepsilon/N$. Denote $R=\sup_{z\in Z}\diam(A_z)$. We then conclude that
$$f^{-1}({r(\supp_{\varepsilon}(Ad_{W_f}(T)))})\subseteq B(r(\supp_{\varepsilon/N}(T)),R).$$
By the definition of an invariant subset, we have $\overline{B(r(\supp_{\varepsilon/N}(T)),R)}\subseteq U$ since $\overline{r(\supp_{\varepsilon/N}(T))}\subseteq U$.
As a result, $\overline{r(\supp_{\varepsilon}(Ad_{W_f}(T)))}\subseteq f(U)\subseteq V$. We conclude that $Ad_{W_f}(T)\in\II[Y, V]$. For the ghostly ideal, the proof is parallel, thus omitted.
\end{proof}

\begin{Cor}\label{cor: K-coarse invariant}
If $f: X\to Y$ is a coarse equivalence between the pair $(X,U)$ and $(Y,V)$, then
$$(Ad_{W_f})_*: K_*(\I(X, U))\to K_*(\I(Y,V))\quad\text{and}\quad (Ad_{W_f})_*: K_*(\G(X, U))\to K_*(\G(Y,V))$$
does not depend on the choice of $W_f$ that covers $f$. As a result, $(Ad_{W_f})_*$ is always an isomorphism.
\end{Cor}

\begin{proof}
For any two covering isometries $W_1,W_2$ for $f$, it suffices to prove that the $*$-homomorphisms from $\I(X,U)$ to $M_2(\I(Y,V))$ defined by
$$\alpha_1:T\mapsto\begin{pmatrix}W_1TW_1^*&0\\0&0\end{pmatrix}\quad\text{and}\quad\alpha_2:T\mapsto\begin{pmatrix}0&0\\0&W_2TW_2^*\end{pmatrix}$$
agree on the level of $K$-theory. Take the unitary
$$U=\begin{pmatrix}1-V_1V_1^*&V_1V_2^*\\V_2V_1^*&1-V_2V_2^*\end{pmatrix}.$$
Conjugation by $U$ induces the identity map on $K$-theory, and one can easily check that $U\alpha_1(T)U^*=\alpha_2(T)$. This completes the proof. 

Moreover, $W_f$ can be chosen to be a unitary, in which case $Ad_{W_f}$ is a $*$-isomorphism. This proves that $(Ad_{W_f})_*$ is always an isomorphism.
\end{proof}

\subsection{Reduction to sparse spaces}

In this subsection, we shall reduce \Cref{thm: main theorem} to the case of sparse spaces. 

By \cite[Section 4]{CWY2013}, one can find an $\omega$-excisive cover $\{Y,Z\}$ of $X$ such that $Y$, $Z$, and $Y\cap Z$ are all \emph{sparse}, i.e., they are all coarse disjoint unions of finite spaces. From the definition of Stone-\v{C}ech compactification, we conclude that $\beta X = \beta Y \cup \beta Z$. Then the canonical inclusion $\iota: Y\to X$ induces a canonical inclusion $\iota: G(Y)\to G(X)$. Actually, one can easily check the identification $G(Y)=G(X)|^{\beta Y}_{\beta Y}$.

If $U\subseteq \beta X$ is an invariant open subset, then it follows from the definition that $U^{\beta Y}=U\cap \beta Y$ and $U^{\beta Z}=U\cap\beta Z$ are $G(Y)$-invariant and $G(Z)$-invariant, respectively. In this case, we shall denote $\I(Y,U^{\beta Y})$ and ${\G}(Y,U^{\beta Y})$ the corresponding geometric and ghostly ideal of $C^*(Y)$ with respect to $U^{\beta Y}$. We have the following observation.

\begin{Lem}\label{lem: ideal on subalgebra is still an ideal}
With notations as above, we have that
$$\I(Y, U^{\beta Y})=\I(X,U)\cap C^*(Y)\quad\text{and}\quad \G(Y, U^{\beta Y})=\G(X,U)\cap C^*(Y).$$
\end{Lem}

\begin{proof}
For the geometric ideal, it suffices to show
$$\II[Y, U^{\beta Y}]=\II[X,U]\cap \IC[Y]$$
since $\II[Y, U^{\beta Y}]$, $\II[X,U]$, $\IC[Y]$ are respectively dense in $\I(Y, U^{\beta Y})$, $\I(X,U)$, $C^*(Y)$. For any $T\in\II[Y, U^{\beta Y}]$, it is direct to see that $r(\supp_{\varepsilon}(T))\subseteq Y\cap U$ from definition. Thus, $T\in \II[X,U]$. On the other hand, if $T\in\IC[Y]$ satisfies that $\overline{r(\supp_{\varepsilon}(T))}\subseteq \beta Y\cap U$, then $r(\supp_{\varepsilon}(T))\subseteq Y\cap U$. This shows that $T\in \II[X,U]$. For the ghostly ideal, the proof is parallel, omitted.
\end{proof}

Write $V_Y$ the $G(X)$-invariant open subset in $\beta X$ generated by $Y$. Notice that $\I(X, V_Y)$ is isomorphic to the ideal of $C^*(X)$ generated by $Y$, denoted by $C^*(X,Y)$ as in \cite{HRY1993}. Thus, one has that
$$\I(X,V_Y)=\lim_{r\to\infty}C^*(B(Y_R)).$$
As a result, there exist canonical inclusions
$$\I(Y,U^{\beta Y})\hookrightarrow C^*(Y)\hookrightarrow \I(X,V_Y).$$

\begin{Lem}\label{lem: ideal and subalgebra cap I}
With notation as above, the canonical inclusion map $\iota: C^*(Y)\cap \I \to\I(X, V_Y)\cap \I$ induces an isomorphism in $K$-theory
$$\iota_*: K_*(C^*(Y)\cap \I)\xrightarrow{\cong} K_*(\I(X, V_Y)\cap \I),$$
where $\I=\I(X,U)$ or $\I=\G(X,U)$.
\end{Lem}

\begin{proof}
As we discussed above, the inclusion map can be induced by the inductive limit of the inclusions as $R$ tends to infinity
$$\iota_{RR'}: C^*(Y_R)\cap\I\to C^*(Y_{R'})\cap\I.$$
For any $R,R'$, the canonical inclusion $\iota: (Y_R, U^{\beta Y_R})\to (Y_{R'}, U^{\beta Y_{R'}})$ is a coarse equivalence, since $Y_R$ is a net of $Y_{R'}$. By \Cref{cor: K-coarse invariant} and \Cref{lem: ideal on subalgebra is still an ideal}, it induces an isomorphism on $K$-theory, i.e.,
$$(\iota_{RR'})_*: K_*(C^*(Y_R)\cap\I)\xrightarrow{\cong} K_*(C^*(Y_{R'})\cap\I)$$
is an isomorphism for any $R,R'\geq 0$. As a result, the inclusion $\iota$ induces an isomorphism on $K$-theory, i.e.,
$$\iota_*: K_*(C^*(Y)\cap \I)\xrightarrow{\cong} K_*(\I(X, V_Y)\cap \I).$$
This completes the proof.
\end{proof}

\begin{Pro}\label{pro: reduction to sparse space}
Let $X$ be a metric space with bounded geometry. 
Fix a $G(X)$-invariant open subset $U\subseteq \beta X$. For any subspace $Y\subseteq X$, denote $U^{\beta Y}=\beta Y\cap U$. If the inclusion $i: \I(Y, U^{\beta Y})\to\G(Y, U^{\beta Y})$ induces an isomorphism in $K$-theory for any sparse subspaces $Y\subseteq X$,
then the inclusion $i:\I(X,U)\to\G(X,U)$ induces an isomorphism in $K$-theory, i.e.,
$$i_*: K_*(\I(X,U))\to K_*(\G(X,U))$$
is an isomorphism.
\end{Pro}

\begin{proof}
Take $\{Y,Z\}$ to be an $\omega$-excisive cover of $X$ such that $Y,Z$ and $Y\cap Z$ are all sparse. Since $Y\cup Z=X$, we have $V_Y\cup V_Z=\beta X$ by the definition of the Stone-\v{C}ech compactification, where $V_Y$ and $V_Z$ are $G(X)$-invariant open subsets generated by $Y$ and $Z$, respectively. It is proved in \cite{HRY1993} that
$$\I(X,V_Y)+\I(X,V_Z)=C^*(X)\quad\text{and}\quad \I(X,V_Y)\cap\I(X,V_Z)=\I(X,V_{Y\cap Z}).$$
The second equation follows from the fact that $\{Y, Z\}$ is $\omega$-excisive. It is direct to see that $V_Y\cap V_Z=V_{Y\cap Z}$.

For an invariant open subset $U \subseteq \beta X$, let $\I$ denote either the geometric ideal $\I(X,U)$ or the ghostly ideal $\G(X,U)$. By taking intersections with $\I$ in the pushout diagram of the geometric and ghostly ideals mentioned above, we obtain the following pushout diagram:
\[\begin{tikzcd}
\I(X,U_{Y\cap Z})\cap \I \arrow[d] \arrow[r] & \I(X,U_{Z})\cap \I \arrow[d] \\
\I(X,U_{Y})\cap \I \arrow[r]           & \I        
\end{tikzcd}\]
By the Five Lemma, it suffices to prove the canonical inclusion map
$$i:\I(X,U_{W})\cap \I(X,U)\to \I(X,U_{W})\cap \G(X,U)$$
induces an isomorphism on $K$-theory for $W=Y,Z$ and $Y\cap Z$ which are all sparse. By \Cref{lem: ideal and subalgebra cap I}, we conclude that $i_*$ is identified with
$$i_*: K_*(C^*(W)\cap \I(X,U))\to K_*(C^*(W)\cap \G(X,U)).$$
By \Cref{lem: ideal on subalgebra is still an ideal}, the map $i_*$ is further identified with
$$i_*: K_*(\I(W,U^{\beta W}))\to K_*(\G(W,U^{\beta W})).$$
Therefore, it suffices to prove $i_*$ is an isomorphism for all sparse subspaces of $X$, as the general case will then follow. This completes the proof of the proposition.
\end{proof}

\begin{Rem}
For ghostly ideals, we actually have the parallel pushout diagram, i.e.,
$$\G(X,V_Y)+\G(X,V_Z)=C^*(X)\quad\text{and}\quad \G(X,V_Y)\cap\G(X,V_Z)=\G(X,V_{Y\cap Z}).$$
Since $\I(X,V_Y)\subseteq \G(X,V_Y)$, the first equation follows from that $\I(X,V_Y)+\I(X,V_Z)=C^*(X)$. For the second one, take an element $T\in C^*(X)$ and an invariant open subset $V\subseteq \beta X$. Define a function $f_T: X\times X\to\IR$ by
$$f_T(x,y)=\|T_{xy}\|.$$
By definition, $f_T$ admits an extension to a function in $C_0(G(X))$. The reasoning is that $T$ can be approximated by finite propagation operators. Since any such operator is supported in a set of the form $\overline{\Delta_R}$, its associated function lies in $C_c(G(X))$. It is direct from the definition that $T\in\G(X, V)$ if and only if $r(\supp(f_T))\subseteq V$. As a result, $T\in \G(X,V_Y)\cap\G(X,V_Z)$ if and only if
$$r(\supp(f_T))\subseteq V_Y\cap V_Z=V_{Y\cap Z}.$$
This proves the pushout diagram of ghostly ideals.
\end{Rem}

\section{The $\ell^p$-Bott-Dirac construction and twisted ghostly ideals}\label{sec: proof}

In this section, we prove \Cref{thm: main theorem}. The argument proceeds in the following three steps.

\subsection{Step 1. On the Bott-Dirac operator for $\ell^p$-spaces}

In the following subsections, we shall prove \Cref{thm: main theorem} in several steps. 
First of all, we shall briefly recall the definition of the Bott-Dirac operator in finite-dimensional $\ell^p$-spaces. Please refer to \cite{WXYZ2024} for a detailed introduction.

Let $E=\IC^n$ be an $n$-dimensional vector space with a canonical basis $\{e_1,\cdots,e_n\}$. For any $p\geq 1$, we can equip it with the $\ell^p$-norm associated with the given basis $\|\cdot\|_p$. Denote by $S_p(E)$ the sphere of $E$ under the norm $\|\cdot \|_p$. Consider the Mazur map
$$\psi: S_p(E)\to S_2(E),\quad \sum_{i=1}^na_ie_i\mapsto \sum_{i=1}^na_i|a_i|^{p/2-1}e_i.$$
We can extend the Mazur map homogeneously to
$$\Psi: (E,\|\cdot\|_p)\to (E,\|\cdot\|_2),\quad v\mapsto \left\{\begin{aligned}&\|v\|_p\psi\left(\frac{v}{\|v\|_2}\right),&&\text{if }v\ne 0\\&0,&&\text{if }v=0\end{aligned}\right..$$

Denote by $\Cl(E)$ the Clifford algebra associated with $(E,\|\cdot\|_2)$. It is easy to verify that if $E$ is an $n$-dimensional space, then $\Cl(E)$ is a vector space of dimension $2^n$. Given an orthonormal basis $\{e_1, \dots, e_n\}$ of $E$, the set
\[\{ e_{i_1} e_{i_2} \cdots e_{i_k} \mid 1 \le i_1 < i_2 < \cdots < i_k \le n,\ 0 \le k \le n \},\]
forms a basis for $\Cl(E)$. This basis then defines a Hilbert space structure on $\Cl(E)$, which we denote by $\mathfrak{h}_E$. Moreover, $\Cl(E)$ acts canonically on $\mathfrak{h}_E$ by left multiplication.

For any $v\in (E,\|\cdot\|_p)$, we define the Bott generator to be the unbounded function
$$C_v: (E,\|\cdot\|_p)\to\Cl(E),\quad w\mapsto \Psi(w-v)$$
Define the Dirac operator $D$ to be
$$D=\sum_{i=1}^n\widehat{e_i}\frac{\partial}{\partial x_i},$$ 
where $\wh{v}:\Cl(E)\to\Cl(E)$ is defined to be $\wh{v}(w)=(-1)^{deg(w)}wv$ for any $v\in E$ and $w\in\Cl(E)$. Here, $D$ is viewed as an unbounded operator on $L^2(E,\mathfrak{h}_E)$

For any $s\geq 1$, we define the \emph{Bott-Dirac operator} to be
$$B_{s,v}=s^{-1}D+C_v$$
We summarize the facts we will need about the Bott-Drac operator in the next result. A detailed proof for the proposition can be found in \cite[Lemma 2.7]{WXYZ2024}:

\begin{Pro}[\cite{WXYZ2024}]\label{pro: Fredholm index of Bott-Dirac}
For any $s > 1$ and $v \in E$, the operator $B_{s, v}$ is essentially self-adjoint on $\mathfrak{h}_E$ and admits a compact resolvent. Furthermore, it is a Fredholm operator with index $1$.
\end{Pro}

In \cite{WXYZ2024}, the identity map $f(v)=v$ used in the construction of the Clifford operator in \cite[Section 12.1]{HIT2020} is replaced by the extended Mazur map. Consequently, the Bott-Dirac operator remains well-behaved. This is primarily due to the fact that $\Psi$ is a Lipschitz map, see \cite[Lemma 2.6]{WXYZ2024}.

\begin{Lem}\label{lem: Mazur is Lip}
Assume $p\geq  2$. The extended Mazur map $\Psi: (E,\|\cdot\|_p)\to (E,\|\cdot\|_2)$ is Lipschitz, i.e.,
$$\|\Psi(v) -\Psi(u)\|_2 \leq L\|v-u\|_p,$$
for any $u,v\in (E, \|\cdot\|_p)$, where $L = 1+p\cdot 2^p$.
\end{Lem}

It is worth noting that \Cref{lem: Mazur is Lip} is crucial for the spectral analysis of $B_{s,v}$. This lemma is also essential for controlling the support of the \emph{bounded Bott-Dirac operator} $F_{s,v}$, which we define below and will use frequently throughout the paper.

\begin{Def}\label{def: bounded Bott-Dirac operator}
For any $v\in E$ and $s>1$, we define the \emph{bounded Bott-Dirac operator} to be
$$F_{s,v}:=B_{s,v}(1+B^2_{s,v})^{-\frac 12}.$$
\end{Def}

Notice that $F_{s,v}$ is a bounded operator on $L^2(E,\mathfrak{h}_E)$. There is a natural action of $C_0(E)$ on $L^2(E, \mathfrak{h}_E)$, given by pointwise multiplication. Consequently, for any $p \in [1, \infty)$, we can discuss the propagation of $F_{s,v}$ in the $E$-direction with respect to the norm $\|\cdot\|_p$. We shall denote this number by $\prop_{E,p}(F_{s,v})$.

For any $v\in E$ and $R>0$, we denote $\chi_{v,R}:=\chi_{B(v,R)}$ for simplicity.  The detailed properties of $F_{s,v}$ are summarized in \cite[Proposition 12.1.10]{HIT2020} for the case when $p=2$ and \cite[Proposition 2.17]{WXYZ2024} for the case when $p\geq 2$. Here we simplify this proposition as follows.

\begin{Pro}[\cite{HIT2020,WXYZ2024}]\label{pro: summary on Bott-Dirac}
For each $\varepsilon>0$, there exists an odd function $\Phi:\mathbb{R}\to[-1,1]$ with $\lim_{t\to\pm\infty}\Phi(t)=\pm1$, satisfying the following properties:
\begin{enumerate}
\item[(1)] For all $s\in [1,\infty)$ and $v\in E$, $\|F_{s,v}-\Phi(B_{s,v})\|<\varepsilon$. Moreover, there exists $R_0>0$ such that $\prop_{E,p}(\Phi(B_{s,v}))\leq s^{-1}R_0$ for all $s, v$ and $p\ge 2$.
\item[(2)] For all $s\in[1,\infty)$, the operators $\Phi(B_{s,v})^2-1$ and $\Phi(B_{s,v})-\Phi(B_{s,w})$ are compact. Furthermore, the map $v \mapsto \Phi(B_{s,v})$ is Lipschitz continuous with respect to the $p$-norm on $E$, i.e.,
$$\|\Phi(B_{s,v})-\Phi(B_{s,w})\|\leq c\|v-w\|_p.$$
\item[(3)] The map $s\mapsto\Phi(B_{s,v})$ is strongly continuous. Moreover, for any $r>0$, the families
$$\{s\mapsto \Phi(B_{s,v})^2-1\}_{v\in E}\quad\text{ and }\quad \{s\mapsto \Phi(B_{s,v})-\Phi(B_{s,w})\}_{\|v-w\|_p\leq r}$$ are norm equi-continuous on compact subsets of $[1,\infty)$.
\item[(4)] For any $r>0$, there exists $R_{\varepsilon,r}>0$ (independent of $\dim E$) such that for large $R\ge R_{\varepsilon,r}$, $v,w\in E$ with $\|v-w\|_p\leq r$ and $s\in[L\cdot(\operatorname{dim} E_n)^2,\infty)$, the operators satisfy the following norm estimates:
$$\|(\Phi(B_{s,v})^2-1)(1-\chi_{v,R})\|<\varepsilon \quad \text{and} \quad \|(\Phi(B_{s,v})-\Phi(B_{s,w}))(1-\chi_{v,R})\|<\varepsilon.$$
\end{enumerate}
\end{Pro}

\subsection{Step 2. Construction of the twisted Roe algebra}\label{sec: step 2 in section 3}

The first step ensures that, without loss of generality, we may assume $X=\bigsqcup_{n=0}^\infty X_n$ is a sparse space that admits a coarse embedding into $\ell^p$. By \cite[Lemma 12.5.4]{HIT2020}, we can assume the distance between $X_n$ and $X_m$ is infinite whenever $n\ne m$ (called \emph{separated disjoint union of $\{X_n\}_{n\in\IN}$}). In this case, the Hilbert space $\H_X$ has a canonical decomposition $\H_X=\bigoplus_{n\in\IN}\H_{X_n}$. The Roe algebra $C^*(X)$ is therefore a subalgebra of $\prod_{n\in\IN}C^*(X_n)$

The coarse disjoint union of $\{X_n\}_{n\in\IN}$ admits a coarse embedding into an $\ell^p$-space if and only if each $X_n$ admits a coarse embedding into the $\ell^p$-space with respect to the same controlling functions $\rho_+, \rho_-:\IR_+\to\IR_+$. Let $f_n: X_n\to \ell^p$ be a coarse embedding. Then $\mathrm{span}\{f_n(X_n)\}$ forms a finite-dimensional subspace of $\ell^p$, which must be coarsely equivalent to a finite-dimensional $\ell^p$ space. Without loss of generality, we may assume the image of $f_n$ is a finite-dimensional linear space $(E_n,\|\cdot\|_p)$. Set $d_n=\dim(E_n)$. We shall write $E$ to be the collection of all $\{E_n\}$.

Define
$$\H_{X_n,E_n}=\H_{X_n}\ox L^2(E_n,\mathfrak{h}_{E_n})\quad\text{and}\quad \H_{X,E}=\bigoplus_{n\in\IN}\H_{X_n,E_n}.$$
Note that we define the Roe algebra $C^*(X, \H_{X, E})$ with respect to $\H_{X,E}$. We will consider $C^*(X_n, \H_{X_n})$ as represented on $\H_{X_n,E_n}=\H_{X_n}\ox L^2(E_n,\mathfrak{h}_{E_n})$ via the amplified representation $T\mapsto T\ox 1$. Thus $C^*(X, \H_X)$ can be viewed as a subalgebra of the multiplier algebra of $C^*(X,\H_{X, E})$. 

We define $E$ as the disjoint union $\bigsqcup_{n\in\IN} E_n$, (where $C_0(E)$ denotes the $C^*$-algebra of continuous functions on $E$ vanishing at infinity). Thus, $\H_{X,E}$ admits natural actions of $C_0(X)=\bigoplus_{n\in\IN}C_0(X_n)$ and $C_0(E)=\bigoplus_{n\in\IN}C_0(E_n)$.
Hence, for an operator $T$ on this space, we can discuss its propagation in the $X$-direction and its propagation in the $E$-direction, denoted by $\prop_X(T)$ and $\prop_{E,p}(T)$, respectively. Explicitly, an operator $T$ has propagation in the $E$-direction bounded by $R \geq 0$ if for any $f, g \in C_0(E)$ supported on sets separated by a $p$-distance greater than $R$, we have $fTg = 0$.

For any $v\in E_n$, we set $\chi^{(n)}_{R,p,v}$ to be the characteristic function of the ball $B(v,R)$ in $(E_n,\|\cdot\|_p)$. Consider the function $\chi^{(n)}_{R,p}: X\mapsto B(L^2(E_n,\mathfrak{h}_{E_n}))$ defined by
$$\chi^{(n)}_{R,p}(x)=\chi^{(n)}_{R,p,f_n(x)}.$$
Notice that each $\chi^{(n)}_{R,p,v}$ can be viewed as a multiplication operator on $L^2(E_n,\mathfrak{h}_{E_n})$. Thus, $\chi^{(n)}_{R,p}$ can be viewed as a bounded operator on $\H_{X_n,E_n}$.

Now, we are ready to define the twisted algebras:

\begin{Def}\label{twisted Roe algebra}
Let $\prod_{n\in\IN}C_b([1,\infty),C^*(X_n, \H_{X_n,E_n}))$ be the direct product $C^*$-algebra of all bounded continuous functions from $[1,\infty)$ to $C^*(X_n, \H_{X_n,E_n})$. Write elements of this algebra as collection $T=(T_{s,n})_{n\in\IN,s\in[1,\infty)}$. Let $\IA[X,E]$ be the $*$-subalgebra of $\prod_{n\in\IN}C_b([1,\infty),C^*(X_n, \H_{X_n,E_n}))$ consisting of elements satisfying the following conditions
\begin{itemize}
\item[(1)] $\sup\limits_{s\in[1,\infty),n\in\IN}\prop_{X_n}(T_{s,n})<\infty$;
\item[(2)] $\lim\limits_{s\to\infty}\sup\limits_{n\in\IN}\prop_{E_n,p}(T_{s,n})=0$;
\item[(3)] $\lim\limits_{R\to\infty}\sup\limits_{s\in[1,\infty),n\in\IN}\|\chi^{(n)}_{R,p}T_{s,n}-T_{s,n}\|=\lim\limits_{R\to\infty}\sup\limits_{s\in[1,\infty),n\in\IN}\|T_{s,n}\chi^{(n)}_{R,p}-T_{s,n}\|=0$.
\end{itemize}
Define the twisted Roe algebra $A(X,E)$ to be the closure of $\IA[X,E]$ with respect to the norm induced from $\prod_{n\in\IN}C_b([1,\infty),C^*(X_n, \H_{X_n,E_n}))$.
\end{Def}

Next, we will use the bounded Bott-Dirac map from \Cref{def: bounded Bott-Dirac operator} to construct the index map:
$$\ind:K_*(C^*(X,\H_X))\to K_*(A(X,E)).$$
For each $s\in[1,\infty)$, $F_{s,v,E_n}:L^2(E_n,\mathfrak{h}_{E_n})\to L^2(E_n,\mathfrak{h}_{E_n})$ is the bounded operator as in \Cref{def: bounded Bott-Dirac operator}. Set $ s_n := L(\operatorname{dim} E_n)^2$, where $L$ is the Lipschtiz constant in \Cref{lem: Mazur is Lip}.

\begin{Def}\label{def: multiplier F}
Let $F:[1,\infty)\to\B(\H_{X,E})$ be a function
$$s\mapsto F_s=(F_{s,1},\cdots,F_{s,n},\cdots)\in\prod_{n\in\IN}\B(\H_{X_n,E_n})\subseteq\B(\H_{X,E}),$$
where $F_{s,n}\in \B(\H_{X_n,E_n})$ is defined to be
$$(F_{s,n})_{xy}=\left\{\begin{aligned}&F_{s+s_n,f_n(x),E_n},&&\text{if }x=y;\\&0,&&\text{if }x\ne y. \end{aligned}\right.$$
Here, we view each element in $\B(\H_{X_n, E_n})$ as an $X_n$-by-$X_n$ matrix, where each entry of the matrix is a bounded linear operator on $L^2(E_n,\mathfrak{h}_{E_n})$.

Analogously, if $\Phi$ has the properties in \Cref{pro: summary on Bott-Dirac} for some $\varepsilon>0$, we write $F^{\Phi}$, $F^{\Phi}_s$ and $F^{\Phi}_{s,n}$ for the operators built in the same way as $F$, $F_s$ and $F_{s,n}$, but replacing $F_{s,v,E_n}$ with $\Phi(B_{s,v,E_n})$.
\end{Def}

By \Cref{pro: summary on Bott-Dirac}, the function $F$ is strong $*$-continuous with respect to $s \in [1, \infty)$. To introduce the index map, we also need the following lemma. The proof of the lemma can be found in \cite[Proposition 2.19]{WXYZ2024} and \cite[Section 12.3]{HIT2020}. We will omit some details here.

\begin{Lem}\label{lem: index class}\begin{itemize}
\item[(1)] The function $F$ defined above is an odd, self-adjoint, contractive operator in the multiplier algebra of $A(X,E)$.
\item[(2)] The Roe algebra $C^*(X,H_X)$ is a $C^*$-subalgebra of the multiplier algebra of $A(X,E)$.
\item[(3)] For any $T\in C^*(X, H_X)$ and $s\in[1,\infty)$, the function $s\mapsto [T,F_s]$ and $s\mapsto T(F^2_s-1)$ is in $A(X,E)$.
\item[(4)] For any $P\in C^*(X,H_X)$ a projection, the function
$$s\mapsto (PF_sP)^2-P$$
is in the corner $PA(X,E)P$.\qed
\end{itemize}\end{Lem}

It should be noted that in \Cref{def: multiplier F}, the parameter $s$ is shifted to the right by $s_n$ (depending on the dimension) to ensure the validity of the preceding \Cref{lem: index class} applying (4) in \Cref{pro: summary on Bott-Dirac}.

By \cite[Corollary 2.7.4]{HIT2020}, a $K_0$-theory element of $C^*(X)$ can be represented by a single projection $[P]$.
For any projection $P\in C^*(X,\H_X)$, $PF_sP$ is an odd, self-adjoint operator on the graded Hilbert space $L^2[1,\infty)\ox \H_{X, E}$. We can build an index class of $K_0(PA(X, E)P)$ by using $PF_sP$ and \Cref{lem: index class}. We can decompose $PF_sP$ as a matrix with respect to the grading on each $\H_{X,E}$ and $PF_sP$ has the form
$$PF_sP=\begin{pmatrix}0&u^*\\u&0 \end{pmatrix}$$
where $P$ is even and has the form
$$P=\begin{pmatrix}P_0&0\\0&P_1 \end{pmatrix}.$$
We define the index class of $PF_sP$ in $K_0(PA(X,E)P)$, denoted by $\ind(PF_sP)$, to be
$$\ind(PF_sP)=\begin{bmatrix}(P_0-u^*u)^2&u^*(P_1-uu^*)\\u(2P_0-u^*u)(P_0-u^*u)&uu^*(P_1-uu^*) \end{bmatrix}-\begin{bmatrix}0&0\\0&P_1 \end{bmatrix}$$
We define the inclusion homomorphism: $i_P:PA(X,E)P\to A(X,E)$. It induces a group homomorphism $(i_P)_*:K_*(PA(X,E)P)\to K_*(A(X,E))$ on the level of $K$-theory.

Define the index map
$$\ind:K_0(C^*(X, \H_X))\to K_0(A(X,E))$$
by the following formula:
$$[P]\mapsto (i_P)_*(\ind(PF_sP)).$$
By using a suspension argument, we can define the index map on the level of $K_1$. Similarly, we can also define the localization index map for both $K_0$ and $K_1$ by applying the above construction pointwise in $t$.

\begin{Def}\label{index map}
Define the index map as we constructed above:
$$\ind:K_*(C^*(X, \H_X))\to K_*(A(X,E)).$$
\end{Def}

The index map is a well-defined group homomorphism. For each $s\in[1,\infty)$, define $\iota^s:A(X,E)\to C^*(X, H_{X,E})$ to be the evaluation map $(T_s)\mapsto T_s$. The following proposition is proved in \cite[Proposition 2.20]{WXYZ2024}, also see \cite[Section 12.3]{HIT2020}. We shall provide a detailed proof in \Cref{pro: dirac dual dirac}.

\begin{Pro}\label{pro: full DDD}
For any $s\in [1,\infty)$, the composition
$$K_*(C^*(X, \H_X))\stackrel{ind}{\longrightarrow}K_*(A(X,E))\stackrel{\iota^s_*}{\longrightarrow}K_*(C^*(X, \H_{X,E}))$$
is an isomorphism.\qed
\end{Pro}

\subsection{Step 3. Geometric and Ghostly ideals in the twisted setting}

In this subsection, we consider restricting the above Dirac-dual-Dirac construction to the geometric ideal and the ghostly ideal.

\begin{Def}
For any $(T_{s,n})\in A(X,E)$, we shall define the $\varepsilon$-support of $(T_{s,n})$ to be the set
$$\supp_{\varepsilon}((T_{s,n}))=\bigcup_{n\in\IN}\left\{(x,y)\in X_n\times X_n\ \Big|\ \sup_{s\in[1,\infty)}\|(T_{s,n})_{xy}\|\geq\varepsilon\right\}.$$

Fix an invariant open set $U \subseteq \beta X$. The \emph{geometric ideal} of $A(X,E)$ associated with $U$, denoted by $A_{\I(U)}(X,E)$, is defined to be the ideal generated by
$$\IA_{\I(U)}[X,E]=\left\{(T_{s,n})\in\IA[X,E]\ \big|\ \overline{r(\supp_{\varepsilon}((T_{s,n})))}\subseteq U\right\}.$$
The \emph{ghostly ideal} of $A(X,E)$ associated with $U$, denoted by $A_{\G(U)}(X,E)$, is defined to be the algebra
$$A_{\G(U)}(X,E)=\left\{(T_{s,n})\in A(X,E)\ \big|\ \overline{r(\supp_{\varepsilon}((T_{s,n})))}\subseteq U\right\}.$$
\end{Def}

Note that the distinction between the geometric ideal and the ghostly ideal in the twisted algebra here is parallel to the distinction in the original Roe algebra, which is the difference in the order of taking the $\varepsilon$-support and the completion. It is easy to show that both $A_{\I(U)}(X,E)$ and $A_{\G(U)}(X,E)$ can absorb $\IA[X,E]$, and thus they indeed give two ideals in $A(X,E)$.

\begin{Lem}\label{lem: multiplier in A}
Let $T \in \I(X,U)$ (or $T \in \G(X,U)$, respectively). Then for any $s \in [1, \infty)$, the paths $s \mapsto [T, F_s]$ and $s \mapsto T(F_s^2 - 1)$ belong to $A_{\I(U)}(X,E)$ (or $A_{\G(U)}(X,E)$, respectively).
\end{Lem}

\begin{proof}
We will only prove the case for the map $s \mapsto T(F_s^2-1)$ by using (3) in \Cref{pro: summary on Bott-Dirac}, as the argument for $[T, F_s]$ is analogous. By \Cref{lem: index class}, we already have these two functions belong to $A(X,E)$.

Note that every operator on $\mathcal{H}_{X,E}$ can be viewed as an $X$-by-$X$ matrix, where each entry determines an element in $B(L^2(E, \mathfrak{h}_E))$. Then, for any $s \in [1, \infty)$, we have
$$(T(F_s^2-1))_{xy} = T^{(n)}_{xy} \cdot (F_{s,f(x),E_n}^2 - 1),$$
where $x,y\in X_n$. Thus, we have that
$$\|(T(F_s^2-1))_{xy}\|\leq \|T^{(n)}_{xy}\|\cdot \|(F_{s,f(x),E_n}^2 - 1)\|\leq \|T^{(n)}_{xy}\|.$$
As a conclusion,
\begin{equation}\supp_{\varepsilon}(T(F_s^2-1))\subseteq\supp_{\varepsilon}(T).\label{eq: inclusion}\end{equation}

For the case of the geometric ideal, it suffices to show that for $T \in \II[X, U]$, we have $T(F_s^2-1) \in A_{\I(U)}[X, E]$. Then it is direct from the inclusion \eqref{eq: inclusion}. For the case of the ghostly ideal, taking the closure of the range of the $\varepsilon$-support on both sides of \eqref{eq: inclusion} yields $\overline{r(\supp_\varepsilon(T(F_s^2-1)))} \subseteq \overline{r(\supp_\varepsilon(T))} \subseteq U$, which shows that $T(F_s^2-1) \in A_{\G(U)}(X, E)$ for any $T \in \G(X, U)$.
\end{proof}

\begin{Cor}\label{cor: restriction on ideal}
For any projection $P\in \I(X,U)$ (or $\G(X,U)$, respectively), the function
$$s\mapsto (PF_sP)^2-P$$
is in the corner $PA_{\I(U)}(X,E)P$  (or $PA_{\G(U)}(X,E)P$, respectively).
\end{Cor}

\begin{proof}
Notice that
\begin{equation*}\begin{split}
(PF_sP)^2-P&=P(F_sPF_s-P)P\\
&=P(F_sPF_s-F^2_sP+F^2_sP-P)P\\
&=PF_s[P,F_s]P+P(F^2_s-1)P.
\end{split}\end{equation*}
Combining this with \Cref{lem: multiplier in A}, we have that
$$s\mapsto PF_s[P,F_s]P\quad\text{ and }\quad s\mapsto P(F^2_s-1)P$$
are in $A_{\I(U)}(X,E)$ (or $A_{\G(U)}(X,E)$, respectively) if $P$ is in $\I(X,U)$ (or $\G(X,U)$, respectively).
\end{proof}

To distinguish, we denote the geometric ideal and the ghostly ideal of $C^*(X,\H_{X,E})$ as $\I_E(X,U)$ and $\G_E(X,E)$, respectively. By \Cref{cor: restriction on ideal}, we can then restrict the index map on the geometric ideal and ghostly ideal. Thus, we have the following commuting diagram
\begin{equation}\begin{tikzcd}\label{eq: main diagram}
{K_*(\I(X,U))} \arrow[d,"\ind"] \arrow[r] & {K_*(\G(X,U)} \arrow[d,"\ind"] \arrow[r] & {K_*(C^*(X,\H_X))} \arrow[d,"\ind"] \\
{K_*(A_{\I(U)}(X,E))} \arrow[d,"\iota^s_*"] \arrow[r] & {K_*(A_{\G(U)}(X,E))} \arrow[d,"\iota^s_*"] \arrow[r] & {K_*(A(X,E))} \arrow[d,"\iota^s_*"] \\
{K_*(\I_E(X,U))} \arrow[r]           & {K_*(\G_E(X,U)} \arrow[r]           & {K_*(C^*(X,\H_{X,E})}          
\end{tikzcd}\end{equation}
Analogous to \Cref{pro: full DDD}, we also have the following proposition for ideals. We should mention that the proof of \Cref{pro: dirac dual dirac} can also be applied to \Cref{pro: full DDD}.

\begin{Pro}\label{pro: dirac dual dirac}
For any $s\in [1,\infty)$, the composition
$$K_*(\I(X,U))\stackrel{\ind}{\longrightarrow}K_*(A_{\I(U)}(X,E))\stackrel{\iota^s_*}{\longrightarrow}K_*(\I_E(X,U))$$
and
$$K_*(\G(X,U))\stackrel{\ind}{\longrightarrow}K_*(A_{\G(U)}(X,E))\stackrel{\iota^s_*}{\longrightarrow}K_*(\G_E(X,U))$$
are isomorphisms.
\end{Pro}

\begin{proof}
We shall only prove the $K_0$-case, and the $K_1$-case follows from the suspension argument. The proof for $\I(X, U)$ also applies to $\G(X, U)$, thus we shall focus on $\I(X, U)$.

We define the map $\kappa: (E,\|\cdot\|_p)\to (E,\|\cdot\|_p)$ to be 
$$\kappa(v)=\left\{\begin{aligned}&0&&,\|v\|_p\leq 1\\&v-\frac{v}{\|v\|_p}&&,\text{otherwise},\end{aligned}\right.$$
where $\kappa^k$ denotes the $k$-fold composition of the map $\kappa$. To simplify the notation, for fixed $s\in[1,\infty)$, we write $F_{v,n}$ for $F_{s,v,E_n}$. For each $k,n\in\IN$, we define the bounded operator $F^{(k)}_n:\H_{X_n,E_n}\to \H_{X_n,E_n}$ to be
$$(F^{(k)}_n)_{xy}=\left\{\begin{aligned}&1\ox F_{\kappa^k(f_n(x)),n}&&,x=y\in X_n\\&0&&,x\ne y.\end{aligned}\right.$$
Moreover, we define $F^{(\infty)}_n: \H_{X_n,E_n}\to \H_{X_n,E_n}$ by the formula
$$(F^{(\infty)}_n)_{xy}=\left\{\begin{aligned}&1\ox F_{0,n}&&,x=y\in X_n\\&0&&,x\ne y.\end{aligned}\right.$$
Notice that for each $k\in\IN\cup\{\infty\}$, $F^{(k)}=(F^{(k)}_n)_{n\in\IN}$ is in the multiplier of $\I_E(X,U)$ by \Cref{pro: summary on Bott-Dirac}. Then for any projection $P\in \I(X,U)$, we can define an index map
$$\ind^{(k)}:K_*(\I(X,U))\to K_*(\I_E(X,U))$$
as in Definition \ref{index map}.

By definition, one has that $\ind^{(0)}=\iota^s_*\circ\ind:K_*(\I(X,U))\to K_*(\I_E(X,U))$. We claim that $\ind^{(0)}=\ind^{\infty}$. To see this, let
$$\H^{\infty}_{X_n,E_n}=\ell^2(X_{n})\ox \H\ox(\L^2(E_n,\mathfrak{h}_{E_n}))^{\oplus\infty}\quad\text{ and }\quad H^{\infty}_{X,E}=\bigoplus_{n\in\IN}H^{\infty}_{X_n,E_n}$$
The Hilbert space $H^{\infty}_{X,E}$ is also a geometric module of $X=\bigsqcup X_n$. Denote $C^*(X, H^{\infty}_{X,E})$ the corresponding Roe algebra and $\I_E^\infty(X,U)$ its geometric ideal. The top-left corner inclusion $\I(X,U)\to \I_E^\infty(X,U)$ is defined by $T=(T_n)\mapsto (T_n\oplus_{k\geq1} 0)$. 
It induces an isomorphism on the $K$-theory level.

For each $[P]\in K_0(\I(X,U))$\footnote{By \cite[Corollary 2.7.4]{HIT2020}, a $K_0$-theory element of $\I(X,U)$ can be represented by a single projection $[P]$.}, we write $\ind^{(k)}([P])=[(p^{(k)}_0,\cdots,p^{(k)}_n,\cdots)]-[q_0,\cdots,q_n\cdots]$ for each $k\in\IN\cup\{\infty\}$, where $q_n$ does not depend on $k$ but only depends on $[P]$ by definition.  For simplicity of notation, we write
$$0\oplus_{k\geq1}p^{(k)}_n=\begin{pmatrix}0&0&0&\cdots\\0&p^{(1)}_n&0&\cdots\\0&0&p^{(2)}_n&\\\vdots&\vdots&&\ddots\end{pmatrix}\quad\text{and}\quad
0\oplus_{k\geq1}p^{(\infty)}_n=\begin{pmatrix}0&0&0&\cdots\\0&p^{(\infty)}_n&0&\cdots\\0&0&p^{(\infty)}_n&\\\vdots&\vdots&&\ddots\end{pmatrix}$$
for each $n\in\IN$. We now define the two elements in $K_0(\I_E^\infty(X,U))$
$$a=[(0\oplus_{k\geq1}p^{(k)}_0,\cdots,0\oplus_{k\geq1}p^{(k)}_n,\cdots)]-[(0\oplus_{k\geq1}p^{(\infty)}_0,\cdots,0\oplus_{k\geq1}p^{(\infty)}_n,\cdots)]$$
$$b=[(p^{(0)}_0\oplus_{k\geq1} 0,\cdots,p^{(0)}_n\oplus_{k\geq1} 0,\cdots)]-[(p^{(\infty)}_0\oplus_{k\geq1} 0,\cdots,p^{(\infty)}_n\oplus_{k\geq1} 0,\cdots)].$$
We will give some explanation of why $a\in K_0(\I_E^\infty(X,U))$ here. The propagation of $a$ is less than $5\prop_P(P)$ by definition of the index map, then $a$ can be approximated by operators with finite propagation as $P$ can. To see $a$ is locally compact, for each compact set $K\subset X$, there exists $N\in\IN$ such that $X_n\cap K\ne\emptyset$ for all $n>N$. Note that $f_n(X_n)\subset E_n$ is uniformly bounded by $M\in\IN$, i.e. $f_n(X_n)\subseteq B(0,M)$ for all $n\leq N$. Then we have that $\chi_{K}p_n^{(k)}=\chi_Kp^{(\infty)}_n$ for all $k\geq M$ and $n\leq N$ since $\kappa^k(f_n(x))=0$ in this case. It shows that $\chi_Ka$ is a direct sum of only finite compact operators, which is also compact.

Notice that $b$ is the image of $\ind^{(0)}([P])-\ind^{(\infty)}([P])$ under the top-left corner inclusion. It suffices to show $b=0\in K_*(\I_E^\infty(X,U))$. Consider the path $[0,1]\to B(H_{X,E})$, $t\mapsto F^{(k+t)}=(F^{(k+t)}_n)$, where
$$(F^{(k+t)}_n)_{xy}=\left\{\begin{aligned}&1\ox F_{(1-t)\kappa^k(f_n(x))+t\kappa^{k+1}(f_n(x)),n}&&,x=y\in X_n\\&0&&,x\ne y.\end{aligned}\right.$$
By \Cref{pro: summary on Bott-Dirac}, we have that
$$\|F^{(k+t)}_n-F^{(k+t')}_n\|\leq 3|t-t'|$$
for all $n\in\IN$. Thus, $t\mapsto F_n^{(k+t)}$ is equi-continuous as $n$ varies in $\IN$. We can build a path
$$[0,1]\to \I_E^\infty(X,U)\quad t\mapsto a_t=(\oplus_{k\in\IN}(p^{(k+t)}_n))-(\oplus_{k\in\IN}(p^{(\infty)}_n))$$
by using the index map with respect to $F^{(k+t)}$ pointwisely, where
$$\oplus_{k\in\IN}(p^{(k+t)}_n)=\begin{pmatrix}p^{(0+t)}&0&0&\cdots\\0&p^{(1+t)}_n&0&\cdots\\0&0&p^{(2+t)}_n&\\\vdots&\vdots&&\ddots\end{pmatrix}.$$
Notice that $[a_0]=[a+b]\in K_0(\I_E^\infty(X,U))$ and it is direct to see $[a_1]=[a]$ by using a rotation homotopy. Thus $[b]=0$ as claimed.

Finally, we still need to show $\ind^{(\infty)}$ is an isomorphism. Recall that $f(x)=x(1+x^2)^{-\frac 12}$ and $F_{s,0,E_n}=f(s^{-1}D+C)$. For each $n\in\IN$, we define a homotopy $[0,1]\to \K(L^2(E_n,\mathfrak{h}_{E_n}))$ by
$$t\mapsto P_{E_n}(t)=\left\{\begin{aligned}&f^2(t^{-1}((s+s_n)^{-1}D_{E_n}+C_{E_n}))-1&&,t\in(0,1];\\&p_k&&,t=0,\end{aligned}\right.$$
where $p_k$ is the one-dimensional projection onto the kernel of the Bott-Dirac operator $B_{s,0,E_n}$. We denote $[P]=[(P_0,\cdots,P_n,\cdots)]$ and $\ind^{(\infty)}([P])=[(Q_0,\cdots,Q_n,\cdots)]$. By using the homotopy $(P_{E_n}(t))_{t\in[0,1]}$, one can check that $Q_n$ is homotopic to $P\ox p_k$. However, the homotopy is not equi-continuous as the dimension of $E_n$ increases. The reason comes from the fact that we replaced the parameter $s$ by $s+s_n$ for each $E_n$. One can check that $\|F_{E_n}(t)-F_{E_n}(t')\|\leq |f(\frac{2}{t(s+s_n)})-f(\frac{2}{t'(s+s_n)})|\leq (s+s_n)|t-t'|$, where the Lipschitz constant is dependent on $n$.

However, we can use a "stacking argument" as in \cite[Proposition 12.6.3]{HIT2020} to get around this problem. Define $[0,1]\mapsto \I_E(X,U)$ by $t\mapsto P(t)=(P_0(t),\cdots,P_n(t),\cdots)$ where $P_n(t)$ is the homotopy between $Q$ and $P\ox p_k$ for each $n\in\IN$. Let $N_n=[s+2s_n]+1$. For each $t\in[0,1]$, we define $P^N(t)=(P^N_n(t))\in \I_E(X,U)$
$$P^N_n(t)=\begin{pmatrix}P_n(0)&0&0&\cdots\\0&P_n(\frac{1}{N_n})&0&\cdots\\\vdots&\vdots&\ddots&\vdots\\0&0&\cdots&P_n(1)\end{pmatrix}-\begin{pmatrix}0&0&0&\cdots\\0&P_n(\frac{t}{N_n})&0&\cdots\\\vdots&\vdots&\ddots&\vdots\\0&0&\cdots&P_n(\frac{N_n-1}{N_n}+\frac t{N_n})\end{pmatrix}.$$
Notice that $\|F_{E_n}(\frac t{N_n})-F_{E_n}(\frac {t'}{N_n})\|\leq \frac{s+s_n}{N_n}|t-t'|\leq |t-t'|$. Thus $P^N(t)$ is norm-continuous. Note that $[P^N(0)]=[P(0)]$ and we also have $[P^N(1)]=[P(1)]$ by a rotation homotopy. Thus, we have
$$\ind^{(\infty)}([P])=[P\ox p_k]$$
which is an isomorphism.
\end{proof}

\subsection{Step 4. Twisted geometric ideal vs. twisted ghostly ideal}

In the final part of the proof, we show that in the twisted algebra, ghostly elements automatically lie in the geometric ideal; that is, the operations of taking the completion and taking the $\varepsilon$-support commute.

\begin{Lem}\label{lem: twisted geo vs gho}
For any $(T_{s,n})\in A_{\G(U)}(X,E)$ and $\varepsilon>0$, there exists $(T'_{s,n})\in\IA_{\I(U)}[X,E]$ and $S>0$ such that
$$\sup_{s\in[S,\infty),n\in\IN}\|T_{s,n}-T'_{s,n}\|\leq\varepsilon.$$ 
\end{Lem}

\begin{proof}
By \Cref{twisted Roe algebra}, we have that
$$\lim\limits_{R\to\infty}\sup\limits_{s\in[1,\infty),n\in\IN}\|\chi^{(n)}_{R,p}T_{s,n}-T_{s,n}\|=\lim\limits_{R\to\infty}\sup\limits_{s\in[1,\infty),n\in\IN}\|T_{s,n}\chi^{(n)}_{R,p}-T_{s,n}\|=0.$$
For any $\varepsilon>0$, there exists $R>0$ such that
$$\sup\limits_{s\in[1,\infty),n\in\IN}\|\chi^{(n)}_{R,p}T_{s,n}-T_{s,n}\|\leq\frac\varepsilon 2\quad\text{and}\quad\sup\limits_{s\in[1,\infty),n\in\IN}\|T_{s,n}\chi^{(n)}_{R,p}-T_{s,n}\|\leq\frac\varepsilon 2.$$
Without loss of generality, we shall replace $T_{s,n}$ by $\chi^{(n)}_{R,p}T_{s,n}\chi^{(n)}_{R,p}$ as the norm difference between them is at most $\varepsilon$. In this case, for any $x,y\in X_n$, the entry $(T_{s,n})_{xy}$ defines an element in $\K(L^2(E_n,\mathfrak h_{E_n})\ox\H)$ and
\begin{equation}\chi^{(n)}_{R,f(x),p}(T_{s,n})_{xy}\chi^{(n)}_{R,f(y),p}=(T_{s,n})_{xy}.\label{eq: entry information}\end{equation}

Choose a sufficiently large $S>0$ such that $\sup_{s\geq S,n\in\IN}\prop_{E,p}(T_{s,n})<R$. Equivalently, we have
$$\sup_{x,y\in X,s\geq S, n\in\IN}\prop_{E,p}((T_{s,n})_{xy})<R.$$
Since each entry is an element in $\K(L^2(E_n,\mathfrak h_{E_n}) \otimes \H)$, we can therefore discuss their propagation with respect to $(E,\|\cdot\|_p)$. Since $f$ is a coarse embedding, there exists $M>0$ such that $\|f(x)-f(y)\|_p>3R$ whenever $d(x,y)>M$. Thus, whenever $d(x,y)>M$ and $s>S$, we have that
$$(T_{s,n})_{xy}=\chi^{(n)}_{R,f(x),p}(T_{s,n})_{xy}\chi^{(n)}_{R,f(y),p}=\chi^{(n)}_{R,f(x),p}\cdot\chi^{(n)}_{2R,f(y),p}(T_{s,n})_{xy}\chi^{(n)}_{R,f(y),p}.$$
Since $\|f(x)-f(y)\|_p>3R$, we have that $B(f(x),R)\cap B(f(y),R)=\emptyset$. Thus, $\chi^{(n)}_{R,f(x),p}\cdot\chi^{(n)}_{2R,f(y),p}=0$. This means that $(T_{s,n})_{xy}=0$ whenever $d(x,y)>M$ and $s>S$, i.e., $T_{s,n}$ has finite $X$-propagation for large $s$.

Define
$$T'_{s,n}=\left\{\begin{aligned}&T_{s,n},&&\text{when }s\geq S;\\&T_{S,n},&&\text{when }s\leq S.\end{aligned}\right.$$
That is, $T'_{s,n}$ coincides with $T_{s,n}$ for $s > S$, and for $s \le S$, we extend it as a constant function to define it on the entire interval $[1, \infty)$. Note that $T'_{s,n} \in \IA[X,E]$ since it has finite $X$-propagation for all $s$, and by definition, we also have $T'_{s,n} \in A_{\G(U)}(X,E)$. Directly from the definition, it is easy to verify that $\IA[X,E] \cap A_{\G(U)}(X,E) = \IA_{\I(U)}[X,E]$. Therefore, $T'_{s,n}$ satisfies the requirements of the lemma.
\end{proof}

As a conclusion, we have the following result.

\begin{Cor}\label{cor: twisted geo vs gho}
The canonical inclusion $i:A_{\I(U)}(X,E)\to A_{\G(U)}(X,E)$ induces an isomorphism on $K$-theory, i.e.,
$$i_*: K_*(A_{\I(U)}(X,E))\to K_*(A_{\G(U)}(X,E)).$$
\end{Cor}

\begin{proof}
Since the K-theory of $A_{\I(U)}(X, E)$ and $A_{\G(U)}(X, E)$ only depends on their asymptotic behavior as $s \to \infty$, \cite[Lemma 6.4.11]{HIT2020}, and \Cref{lem: twisted geo vs gho} shows that these two algebras have the same asymptotic behavior, the conclusion thus follows naturally. We leave the details to the reader.
\end{proof}

We can now prove \Cref{thm: main theorem}.

\begin{proof}[Proof of \Cref{thm: main theorem}]
The theorem follows directly from the diagram \eqref{eq: main diagram}, \Cref{pro: dirac dual dirac}, \Cref{cor: twisted geo vs gho}, and a diagram chasing argument.
\end{proof}

\begin{Rem}\label{rem: so do Lp}
For $p\in[1,\infty)$. Let $L^p=L^p[0,1]$ denote the $L^p$-space on $([0,1],m)$ equipped with the canonical Lebesgue measure $m$. By \cite[Proposition III.A.1]{Banachspace}, any separable $L^p$-space $L^p(\Omega,\mu)$ is isometric to the $p$-direct sum of $L^p$ and $\ell^p$. Thus, $L^p[0,1]$ serves as one of the most representative $L^p$-spaces.

We consider the standard filtration of finite-dimensional subspaces given by dyadic step functions. For each $n \in \mathbb{N}$, let $\mathcal{D}_n = \{ [k2^{-n}, (k+1)2^{-n}) : 0 \le k < 2^n \}$ be the partition of $[0,1]$ into dyadic intervals. Define $W_n$ to be the subspace spanned by the characteristic functions of intervals in $\mathcal{D}_n$. Then $W_n$ is linearly isometric to $\ell^p_{2^n}$, and the union $W = \bigcup_{n} W_n$ is dense in $L^p$.

Suppose that $X = \bigsqcup_{n} X_n$ is a sparse space that admits a uniformly coarse embedding $\{f_n: X_n \to L^p\}$. Since $W$ is dense in $L^p$, without loss of generality, we may assume that $X$ coarsely embeds into $W = \bigcup_{n\in\mathbb{N}} W_n$. Since $X_n$ is a finite set, its image $f_n(X_n)$ is a finite subset of $W$, there exists a sufficiently large integer $k_n$ such that $f_n(X_n) \subseteq W_{k_n}$. Identifying $W_{k_n}$ with the subspace of $\ell^p$ supported on the first $2^{k_n}$ coordinates, we see that $f_n$ effectively embeds $X_n$ into $\ell^p$. Thus, the existence of a coarse embedding into $L^p$ implies the existence of a coarse embedding into $\ell^p$. By \cite[Proposition III.A.1]{Banachspace}, any separable $L^p(\Omega, \mu)$ is isometric to the direct sum of $L^p[0,1]$ and $\ell^p$. Thus, our main result holds if we replace $\ell^p$ by any general separable $L^p$-spaces.
\end{Rem}

\Cref{rem: so do Lp} leads us to the following theorem:

\begin{Thm}\label{thm: so do Lp}
The conclusion of \Cref{thm: main theorem} holds true if the assumption of coarse embeddability into $\ell^p$ is replaced by coarse embeddability into an $L^p$-space ($1 \le p < \infty$).\qed
\end{Thm}

\begin{Rem}
It is worth noting that the assumption of separability in \Cref{rem: so do Lp} is actually not necessary. We can provide an alternative proof based on the finite representability of general $L^p$-spaces in $\ell^p$, as discussed in \cite[Section III.E.15]{Banachspace}. Recall that a Banach space $X$ is finitely representable in $Y$ if for any finite-dimensional subspace $E \subseteq X$ and any $\epsilon > 0$, there exists a subspace $F \subseteq Y$ and an isomorphism $T: E \to F$ such that $\|T\| \cdot \|T^{-1}\| < 1 + \epsilon$.

The essential reason why any $L^p$-space (possibly non-separable) is finitely representable in $\ell^p$ lies in the approximation by simple functions. Specifically, any function in a finite-dimensional subspace of $L^p$ can be approximated by simple functions, which are linear combinations of characteristic functions over disjoint measurable sets. Since the number of basis vectors in the subspace is finite, the approximation involves only finitely many such measurable sets. The characteristic functions on these disjoint sets naturally span a finite-dimensional subspace isometric to $\ell^p_N$. Consequently, the linear map embedding (from basis to basis) the finite-dimensional subspace of $L^p$ into the generated $\ell^p_N$ is an almost isometry. This implies that the family of all finite-dimensional subspaces of any $L^p$-space admits a uniformly coarse embedding into $\ell^p$. Thus, coarse embeddability into a general $L^p$-space implies coarse embeddability into $\ell^p$ for sparse spaces, extending our main result to the non-separable setting.
\end{Rem}

\section{Applications to coarse Baum-Connes conjectures and $ONL_{\P_{Fin}}$}\label{sec: applications}

In this section, we present some corollaries of our main theorem and some remarks.

\subsection{Applications}

In \cite{GWZ2025}, a relative version of the coarse Baum-Connes conjecture was introduced. In fact, our main result has the following corollary:

\begin{Cor}\label{cor: relative CBC}
Let $X$ be a metric space with bounded geometry that admits a coarse embedding into an $\ell^p$-space ($p\geq 1$). For any subspace $Y$, the relative coarse Baum-Connes conjecture holds for $(X, Y)$. In particular, the boundary coarse Baum-Connes conjecture holds for $X$.
\end{Cor}

For the convenience of the reader, we briefly recall the definition of the relative coarse Baum-Connes conjecture. Let $X$ be a metric space with bounded geometry and let $Y \subseteq X$ be a subspace. Let $C^*(X)$ be the Roe algebra of $X$ and $\G(X,U_Y)$ be the ghostly ideal generated by $Y$. We define the relative Roe algebra of the pair $(X, Y)$ as the quotient algebra
$$C^*_{Y,\infty}(X) = C^*(X) / \G(X,U_Y).$$
For $d \ge 0$, we define $P_d(X)$ as the Rips complex of $X$ at scale $d$. Its metric is defined as the hemispherical metric given in \cite[Definition 7.2.8]{HIT2020}. Let $C^*_L(P_d(X))$ be the localization algebra associated with $P_d(X)$. Define $C^*_{L,Y}(P_d(X))$ as the subalgebra of $C^*_L(P_d(X))$ consisting of all functions taking values in $\G(P_d(X),U_Y)$. This is naturally an ideal. We consider the quotient
$$C^*_{L,Y,\infty}(P_d(X)) := C^*_L(P_d(X)) / C^*_{L,Y}(P_d(X)).$$
Then the evaluation map at zero,
$$\text{ev}: C^*_{L,Y,\infty}(P_d(X)) \to C^*_{Y,\infty}(P_d(X)), \quad f \mapsto f(0),$$
defines an index map
$$\text{ev}_*: K_*(C^*_{L,Y,\infty}(P_d(X))) \to K_*(C^*_{Y,\infty}(P_d(X))).$$
By \Cref{cor: K-coarse invariant}, the right-hand side is a coarse equivalence invariant, and thus is isomorphic to $K_*(C^*_{Y,\infty}(X))$. The \emph{relative assembly map} for $(X, Y)$ is defined as the inductive limit of these index maps as $d \to \infty$:
$$\mu: \lim_{d\to\infty} K_*(C^*_{L,Y,\infty}(P_d(X))) \to K_*(C^*_{Y,\infty}(X)).$$
If the subspace $Y$ is chosen to be a bounded set, then the relative coarse Baum-Connes conjecture is called the \emph{boundary coarse Baum-Connes conjecture}, see \cite{FSW2014}.

\begin{proof}[Proof of \Cref{cor: relative CBC}]
Consider the following exact sequence:
\[\begin{tikzcd}
{K_i(A_{Y,L,d})} \arrow[r] \arrow[d, "(1)"] & {K_i(A_{L,d})} \arrow[d, "(2)"] \arrow[r] & {K_i(A_{L,Y,\infty,d})} \arrow[d, "(3)"] \arrow[r] & {K_{i+1}(A_{Y,L,d})} \arrow[r] \arrow[d, "(4)"] & {K_{i+1}(A_{L, d})} \arrow[d, "(5)"] \\
{K_i(A_Y)} \arrow[r]           & {K_i(A)} \arrow[r]           & {K_i(A_{Y,\infty})} \arrow[r]           & {K_{i+1}(A_Y)} \arrow[r]           & {K_{i+1}(A)}
\end{tikzcd}\]
where $A_{Y,L,d} = C^*_{L,Y}(P_d(X))$, $A_{L,d} = C^*_{L}(P_d(X))$, $A_{L,Y,\infty,d} = C^*_{L,Y,\infty}(P_d(X))$, $A_Y = \G(X,U_Y)$, $A = C^*(X)$, and $A_{Y,\infty} = C^*_{Y,\infty}(X)$. 
Since $X$ admits a coarse embedding into $\ell^p$, by \cite{WXYZ2024}, we have that the maps (2) and (5) are isomorphisms for both $i=0,1$. By \Cref{thm: main theorem}, $K_*(\G(X,U_Y))\cong K_*(\I(X,U_Y))\cong K_*(C^*(Y))$. Thus, the maps (1) and (4) are identified with the coarse Baum-Connes assembly map for $Y$, which are also isomorphisms since $Y$ also admits a coarse embedding into $\ell^p$. Then the corollary follows from the Five Lemma.
\end{proof}

The following definition is introduced by M.~Braga, I.~Farah, and A.~Vignati in \cite{BFV2024}.

\begin{Def}[\cite{BFV2024}]
\begin{itemize}
\item[(1)] For any $C^*$-algebra $A$, we shall denote $\P_{Fin}(A)$ the set of all finite rank projections in $A$.
\item[(2)] Let $\I$ be any index set. The \emph{uniform direct product} of $C^*_u(X)$, denoted by $\prod^u_{\I}C^*_u(X)$ is the subalgebra of $\prod_{\I}C^*_u(X)$ generated by all families $(T_i)_{i\in\I}$ with $\sup_{i\in\I}\prop(T_i)<\infty$. One can similarly define the uniform direct product of $C^*(X)$.
\item[(3)] The space $X$ is said to have \emph{operator norm localization for equi-approximable finite-rank projections}, (abbrev. $ONL_{\P_{Fin}}$), if for any $\varepsilon>0$, index set $\I$, and element $(T_i)\in\P_{Fin}(\prod^u_{\I}C^*_u(X))$, there exists $S>0$ such that for any $i\in\I$, there exists $\xi\in\ell^2(X)$ satisfying that
$$\diam(\supp(\xi))\leq S\quad \text{and} \quad \|T_i\xi\|\geq (1-\varepsilon)\|T_i\|$$ 
\item[(4)] If we replace $\P_{Fin}(\prod^u_{\I}C^*_u(X))$ by $\prod^u_{\I}C^*_u(X)$, then we shall say $X$ has \emph{operator norm localization}, (abbrev. ONL, see \cite{CTWY2008}).
\end{itemize}
\end{Def}

We should mention that a large class of spaces satisfies the property $ONL_{\P_{Fin}}$, including all metric spaces that admit a coarse embedding into Hilbert space, see \cite[Lemma 35]{FS2014} and \cite[Theorem 1.3]{BFV2024}.  Our result extends the class to spaces that are coarsely embeddable into $\ell^p$-spaces.

\begin{Cor}\label{cor: CElp implies ONLFin}
If $X$ admits a coarse embedding into $\ell^p$, then $X$ has $ONL_{\P_{Fin}}$.
\end{Cor}

To prove \Cref{cor: CElp implies ONLFin}, we need the following lemma (see also the proof of \cite[Proposition 35]{FS2014}).

\begin{Lem}\label{lem: when ghost notequal cpt}
Let $X$ be a sparse metric space with bounded geometry. All ghost projections in $C^*(X)$ are compact if and only if the canonical inclusion $i:\K\to \G$ induces an isomorphism $i_*:K_0(\K)\to K_0(\G)$, where $\G=\G(X,X)$ is the ghostly ideal associated with the invariant set $X\subseteq \beta X$.
\end{Lem}

\begin{proof}
The $(\Rightarrow)$ part. For any projection $P\in M_n(\G)$, since the Roe algebra is quasi-stable, there exists a projection $P'\in\G$ such that $[P]=[P']\in K_0(\G)$. By assumption, $P'$ is compact, thus $[P']\in K_0(\K)$ and $i_*([P'])=[P]$. Thus $i_*$ is surjective. Since $X$ is sparse, the map $i_*$ is injective by \cite[Proposition 2.10 \& Remark 2.12]{OOY2009}. This shows that $i_*$ is an isomorphism.

The $(\Leftarrow)$ part. Assume for a contradiction that $P\in C^*(X)$ is a non-compact ghost projection. Say $X=\bigsqcup_{n\in\IN}X_n$ and $N_n=\#X_n$.
Notice that $\prod^u_{n\in\IN}C^*(X_n)$ is a subalgebra of $C^*(X)$ and $\K(\H_X)\cap \prod^u_{n\in\IN}C^*(X_n)=\bigoplus_{n\in\IN}M_{N_n}(\K)$. Moreover, the quotient algebras are canonically isomorphic to each other, i.e.,
$$\frac{C^*(X)}{\K(\H_X)}\cong\frac{\prod^u_{n\in\IN}C^*(X_n)}{\bigoplus_{n\in\IN}M_{N_n}(\K)}.$$
Since each $X_n$ is finite, the Roe algebra $C^*(X_n)$ is the algebra of compact operators. Thus, we can define $\tau: C^*(X)\to\prod_{n\in\IN}\K/\bigoplus_{n\in\IN}\K$ via the composition
$$C^*(X)\xrightarrow{\pi}\frac{C^*(X)}{\K(\H_X)}\cong \frac{\prod^u_{n\in\IN}C^*(X_n)}{\bigoplus_{n\in\IN}M_{N_n}(\K)}\hookrightarrow\frac{\prod_{n\in\IN}\K}{\bigoplus_{n\in\IN}\K}.$$
This leads to a ``trace map''
$$\tau_*: K_0(C^*(X))\to K_0\left( \frac{\prod_{n\in\IN}\K}{\bigoplus_{n\in\IN}\K}\right)\cong\frac{\prod_{n\in\IN}\IZ}{\bigoplus_{n\in\IN}\IZ}.$$

For any $T\in C^*(X)$, it is direct to see that $T$ is compact if and only if $\tau(T)=0$. For the non-compact ghost projection $P$, the $\tau(P)$ defines a non-zero projection on $\prod_{n\in\IN}\K/\bigoplus_{n\in\IN}\K$. Write $\tau(P)=[p_1,\cdots,p_n,\cdots]$ with each $p_i$ is a projection in $\K$. Since $\tau(P)$ is non-zero, we have that $\limsup_{n\to\infty}\|p_n\|=1$. Thus, $\tau_*([P])$ is non-zero.

Consider the following commuting diagram
$$\begin{tikzcd}
{K_0(\K)} \arrow[d,"i_*"'] \arrow[r,"\tau_*"] & {\frac{\prod_{n\in\IN}\IZ}{\bigoplus_{n\in\IN}\IZ}} \\
{K_0(\G)} \arrow[ru,"\tau_*"']          &   
\end{tikzcd}$$
where $\G=\G(X,X)$ is the ghostly ideal associated with the invariant set $X\subseteq \beta X$. If $X$ admits a non-compact ghost $P\in C^*(X)$, then $[P]\in K_0(\G)$ can never be in the image of $i_*$ since $\tau_*$ maps $K_0(\K)$ to $0$ while $\tau_*([P])\ne 0$. This leads to a contradiction to the fact that $i_*$ is an isomorphism.
\end{proof}

\begin{proof}[Proof of \Cref{cor: CElp implies ONLFin}]
Assume for a contradiction that $X$ does not have $ONL_{\P_{Fin}}$. By \cite[Theorem 1.3]{BFV2024}, there exists a sparse subspace $X'\subseteq X$ such that there exists a non-compact ghost projection in $C^*_u(X')$. 
Without loss of generality, we may assume that $X$ itself is sparse and $P\in C^*_u(X)$ is a non-compact ghost projection. Take $p\in \K(\H)$ to be a rank one-projection. Then the tensor product $P\ox p$ defines a non-compact ghost projection in $C^*(X)$. By \Cref{lem: when ghost notequal cpt}, the map $i_*$ is not an isomorphism on $K_0$.  This leads to a contradiction to \Cref{thm: main theorem}.
\end{proof}

\subsection{A remark on maximal coarse Baum-Connes conjecture}

In this section, we discuss how the method used in \cite{HIT2020} to prove the coarse Baum-Connes conjecture also applies to the maximal coarse Baum-Connes conjecture. This requires the following proposition.

Let $\IC[X]$ be the algebraic Roe algebra of $X$ on the $X$-module $\H_X=\ell^2(X)\ox \H$, where $\H$ is a separable infinite-dimensional Hilbert space. We shall denote $\B=\B(\H)$ and $\K=\K(\H)$ for simplicity.

\begin{Pro}\label{pro: extension of representation}
Let $\pi: \IC[X]\to \B(\H_\pi)$ be a non-degenerate $*$-representation. Then there exists a representation $\rho:\ell^{\infty}(X,\B)\to\B(\H_\pi)$ on the same Hilbert space $\H_\pi$ such that
$$\rho(\ell^{\infty}(X,\B))\subseteq M(C^*_\pi(X)),$$
where $C^*_{\pi}(X)$ is the $C^*$-completion of $\pi(\IC[X])$ in $\B(\H_\pi)$ and $M(C^*_{\pi}(X))$ is the multiplier algebra of $C^*_{\pi}(X)$.
\end{Pro}

\begin{proof}
Notice that $\ell^{\infty}(X,\K)\subseteq\IC[X]$ is a subalgebra. Moreover $\ell^{\infty}(X,\B)$ is identified with the multiplier algebra of $\ell^{\infty}(X,\K)$. Any $*$-representation of a $C^*$-algebra $A$ can be lifted to a representation of $M(A)$, see \cite[Proposition 1.7.3]{HIT2020}. We shall denote $\rho: \ell^{\infty}(X,\B)\to\B(\H_\pi)$ the lifted representation of $\pi|_{\ell^{\infty}(X,\K)}$.

By \cite[Lemma 3.4]{GWY2008}, any element $T\in\IC[X]$ can be written in the form
$$T=\sum_{i=1}^Nf_iV_i$$
with $f_i\in \ell^{\infty}(X,\K)$ and $V_i$ satisfies that $V^*_iV_i\in\ell^{\infty}(X,\K)$. Thus, for any $g\in \ell^{\infty}(X,\B)$, we have that
$$\rho(g)\pi(T)=\rho(g)\sum\pi(f_i)\pi(V_i)=\sum\pi(gf_i)\pi(V_i)\in\pi(\IC[X]).$$
This means that $\rho(\ell^{\infty}(X,\B))\subseteq M(C^*_\pi(X))$.
\end{proof}

\begin{Lem}\label{lem: split into blocks}
Let $X=\bigsqcup_{n\in\IN}X_n$ be the separated disjoint union of finite spaces $\{X_n\}_{n\in\IN}$. For any representation $\pi:\ell^{\infty}(X)\to\B(H)$, we have that
$\sum_{n\in\IN}\pi(\chi_{X_n})$ SOT-converges to $I$.
\end{Lem}

\begin{proof}
For any $\xi\in H$, by Riesz's representation theorem, the function $\mu:\P(X)\to[0,\|\xi\|]$
$$\mu(E)=\langle \chi_E\xi,\xi\rangle$$
defines a measure on $X$. Thus,
$$\mu(X)=\mu(\bigsqcup_{n\in\IN}X_n)=\sum_{n\in\IN}\mu(X_n).$$
As a result, $\|\sum_{n\geq N}\pi(\chi_{X_n})\xi\|^2=\sum_{n\geq N}\mu(X_n)$ tends to $0$ as $N$ tends to infinity.
\end{proof}

By the above proposition, we can apply the methods from \cite[Chapter 12]{HIT2020} and \cite{WXYZ2024} to the maximal coarse Baum-Connes conjecture. First, consider a maximal representation $\pi$ of the Roe algebra $\IC[X,\H_X]$. By \Cref{pro: extension of representation}, since $\IC[X,\H_X]$ and $\IC[X,\H_{X,E}]$ are unitarily equivalent, this induces a maximal representation $\pi_E$ of $\IC[X,\H_{X,E}]$. One can identify both $\ell^{\infty}(X)$ and $C_0(E)$ as subalgebras of $\prod_{n\in\IN}\ell^{\infty}(X_n, \B(L^2(E_n,\mathfrak{h}_{E_n}) \otimes \H))$, where $L^2(E_n,\mathfrak{h}_{E_n})$ is the Hilbert space introduced in the begining of \Cref{sec: step 2 in section 3}. Therefore, even in the maximal Roe algebra, we can still define the propagation with respect to $X$ and $E$. By \Cref{lem: split into blocks}, we can write an element $T\in C^*_{\max}(X,\H_{X,E})$ as a sequence of $T_n$ with each $T_n=\chi_{X_n}T\chi_{X_n}$. Moreover, since both $\chi_{R,p}$ in \Cref{twisted Roe algebra} and $F_{s,n}$ in \Cref{def: multiplier F} are elements in $\ell^{\infty}(X, \B(L^2(E,\mathfrak{h}_E) \otimes \H))$, the condition (2) in \Cref{twisted Roe algebra} remains well-defined for the maximal Roe algebra.

Recalling the proof of \Cref{lem: index class} in \cite[Chapter 12]{HIT2020}, one finds that the proofs of these statements are, in fact, carried out before completion, namely, within the algebras $\IC[X,\H_X]$ and $\IA[X,E]$. For example, $F$ is a multiplier of $\IA[X,E]$, $\IC[X,\H_X]$ is a subalgebra of the multiplier algebra of $\IA[X,E]$, and so on. These statements, therefore, extend to the maximal completions by \Cref{pro: extension of representation}. Consequently, we can still define the index map from the maximal Roe algebra to the maximal version of the twisted algebra. Moreover, one can verify that the proofs in \cite[Section 12.3 \& 12.4]{HIT2020} also hold within the dense subalgebra $\IC[X,\H_{X,E}]$, and thus they naturally extend to the maximal completions. Hence, combining these observations, we obtain:

\begin{Thm}\label{thm: MCBC}
Let $X$ be a metric space with bounded geometry. If $X$ admits a coarse embedding into an $\ell^p$-space, then the maximal coarse Baum-Connes conjecture holds for $X$.\qed
\end{Thm}

As an application, we have the following result.

\begin{Cor}\label{cor: maximal corollary}
Let $X$ be a metric space with bounded geometry that admits a coarse embedding into an $\ell^p$-space ($p\geq 1$). \begin{itemize}
\item[(1)] The canonical quotient map $\pi: C^*_{\max}(X)\to C^*(X)$ induces an isomorphism in $K$-theory, i.e.,
$$\pi_*: K_*(C^*_{\max}(X))\xrightarrow{\cong} K_*(C^*(X)).$$
\item[(2)] For any subspace $Y$ in $X$, the canonical quotient map $\pi: C^*_{\max,Y,\infty}(X)\to C^*_{Y,\infty}(X)$ induces an isomorphism in $K$-theory, i.e.,
$$\pi_*: K_*(C^*_{\max,Y,\infty}(X))\xrightarrow{\cong} K_*(C^*_{Y,\infty}(X)).$$
\item[(3)] For any subspace $Y$ in $X$, the maximal relative coarse Baum-Connes conjecture holds for $(X, Y)$.
\end{itemize}\end{Cor}

\begin{proof}
(1) Consider the following commuting diagram
\[\begin{tikzcd}
{\lim\limits_{d\to\infty}K_*(C^*_{L,\max}(P_d(X)))} \arrow[r,"\mu_{\max}"] \arrow[d,"(\pi_L)*"'] & {K_*(C^*_{\max}(X))} \arrow[d,"\pi_*"] \\
{\lim\limits_{d\to\infty}K_*(C^*_{L}(P_d(X)))} \arrow[r,"\mu"]           & {K_*(C^*(X))}          
\end{tikzcd}\]
The left vertical map is induced by the canonical quotient map on localization algebras
$$\pi_L: C^*_{L,\max}(P_d(X))\to C^*_{L}(P_d(X))$$
by taking the quotient pointwise. Since both maximal and reduced coarse Baum-Connes conjecture holds for $X$ (by \Cref{thm: MCBC} and \cite[Theorem 1.1]{WXYZ2024}), the maps $\mu_{\max}$ and $\mu$ are isomorphisms. Thus, it suffices to show that $\pi_L$ induces an isomorphism in $K$-theory, this follows from a cutting and pasting argument in \cite[Theorem 3.2]{Yu1997}.

(2) Since the left-hand side of the coarse Baum-Connes conjecture coincides for the maximal and reduced norms, \Cref{thm: MCBC} implies that the map $K_*(C^*_{\max}(X))\to K_*(C^*(X))$ is an isomorphism. We then have the following commutative diagram, 
\[\begin{tikzcd}
{0} \arrow[r] & {C^*_{\max}(X, Y)} \arrow[r] \arrow[d] & {C^*_{\max}(X)} \arrow[r] \arrow[d] & {C^*_{\max,Y,\infty}(X)} \arrow[d] \arrow[r] & {0} \\
{0} \arrow[r] & {\G(X,U_Y)} \arrow[r]           & {C^*(X)} \arrow[r]           & {C^*_{Y,\infty}(X)} \arrow[r]           & {0},
\end{tikzcd}\]
where $C^*_{\max}(X,Y)$ denotes the geometric ideal in the maximal Roe algebra. Its $K$-theory is isomorphic to $K_*(C^*_{\max}(Y))$; see \cite{GWZ2025} for the associated short exact sequence. For the proof of this isomorphism, one may refer to \cite[Lemma 1]{HRY1993} for the reduced case. The argument for the maximal case follows verbatim.

The left vertical map is given by the composition
$$C^*_{\max}(X,Y)\xrightarrow[\text{canonical quotient}]{\pi}\I(X,U_Y)\xrightarrow[\text{canonical inclusion}]{i}\G(X,U_Y).$$
By item (1), $K_*(C^*_{\max}(Y))$ is isomorphic to $K_*(\I(X,U_Y))$ since $Y$ coarsely embeds into $\ell^p$-space, which is further isomorphic to $K_*(\G(X,U_Y))$ by \Cref{thm: main theorem}. Then it follows immediately from the Five Lemma.

(3) It follows directly from (2) and \Cref{cor: relative CBC}.
\end{proof}

\subsection*{Acknowledgement} We thank Jiawen Zhang for insightful comments on the first draft. 
Liang Guo is partially supported by the Chinese Postdoctoral Science Foundation (No. 2025M773059) and the Research Start-up Fund of the Shanghai Institute for Mathematics and Interdisciplinary Sciences (SIMIS). 
Qin Wang is partially supported by NSFC (No. 12571135), Key Laboratory of MEA (Ministry of Education), the Science and Technology Commission of Shanghai (No. 22DZ2229014). 

\subsection*{Conflict of interest} The authors declare that they have no conflict of interest.

\subsection*{Data availability} Data sharing is not applicable to this article as no datasets were generated or analysed during the current study.

\bibliographystyle{alpha}
\bibliography{ref}

\end{document}